\documentclass[11pt]{amsart}
\usepackage{amsmath}
\usepackage{amssymb}
\usepackage{amscd}

\def\NZQ{\mathbb}               
\def\NN{{\NZQ N}}
\def\QQ{{\NZQ Q}}
\def\ZZ{{\NZQ Z}}
\def\RR{{\NZQ R}}

\newtheorem{Theorem}{Theorem}[section]
\newtheorem{Lemma}[Theorem]{Lemma}
\newtheorem{Corollary}[Theorem]{Corollary}
\newtheorem{Proposition}[Theorem]{Proposition}
\newtheorem{Remark}[Theorem]{Remark}

\newtheorem{Definition}[Theorem]{Definition}

\let\epsilon\varepsilon
\let\phi=\varphi
\let\kappa=\varkappa

\begin{document}
\title{Mixed multiplicities of  filtrations}
\author{Steven Dale Cutkosky}
\author{Parangama Sarkar}
\author{Hema Srinivasan}
\thanks{The first author was partially supported by NSF grant DMS-1700046.}
\thanks{The second author was supported by IUSSTF, SERB Indo-U.S. Postdoctoral Fellowship 2017/145 and DST-INSPIRE India}

\address{Steven Dale Cutkosky, Department of Mathematics,
University of Missouri, Columbia, MO 65211, USA}
\email{cutkoskys@missouri.edu}

\address{Parangama Sarkar, Department of Mathematics,
University of Missouri, Columbia, MO 65211, USA}
\email{parangamasarkar@gmail.com}

\address{Hema Srinivasan, Department of Mathematics,
University of Missouri, Columbia, MO 65211, USA}
\email{srinivasanh@missouri.edu}
\begin{abstract}
In this paper we define and explore properties of mixed multiplicities of (not necessarily Noetherian) filtrations of  $m_R$-primary ideals in a Noetherian local ring $R$, generalizing the classical theory for $m_R$-primary ideals. We construct a real polynomial whose coefficients  give the mixed multiplicities. This polynomial  exists if and only if the dimension of the nilradical of the completion of $R$ is less than the dimension of $R$, which holds for instance if $R$ is excellent and reduced.  We show that many of the classical theorems for mixed multiplicities of $m_R$-primary ideals  hold for filtrations, including the famous Minkowski inequalities
 of Teissier,  and Rees and Sharp.
\end{abstract}

\maketitle

\section{Introduction}
The theory of mixed multiplicities of $m_R$-primary ideals in a Noetherian local ring $R$ with maximal ideal $m_R$,  was initiated by Bhattacharya \cite{Bh}, Rees  \cite{R} and Teissier  and Risler \cite{T1}. 
In this paper we extend mixed multiplicities to arbitrary; that is, not necessarily Noetherian, filtrations of $R$ by $m_R$-primary ideals and  explore their properties. 

An account of the history of the Minkowski inequalities of mixed multiplicities is given in \cite{GGV}.  This article explains the origins of this subject in Teissier's work on equisingularity \cite{T1}, and gives many important references.  A  survey of the theory of   mixed multiplicities of  ideals, with proofs,  can be found in  \cite[Chapter 17]{HS}. We refer to this book for references to many important results in this area. We particularly mention Sections 17.1 - 17.3  of \cite{HS} which develops the theory of joint reductions, including discussion of the results of  the papers \cite{R1} of Rees and \cite{S} of  Swanson. A further development is by Katz and Verma \cite{KV} , who generalized mixed multiplicities to ideals which are not all $m_R$-primary. Trung and Verma \cite{TV} computed mixed multiplicities of monomial ideals from mixed volumes of suitable polytopes.  Mixed multiplicities are used by Huh in the analysis of the coefficients of the chromatic polynomial of graph theory in \cite{H}.

The starting point of our investigation is the following theorem which  allows one to define the multiplicity of a filtration of $R$ by $m_R$-primary ideals. As the theorem shows, one must impose the condition that 
the dimension of the nilradical of the completion $\hat R$ of $R$ is less than the dimension of $R$. Let $\lambda(M)$ denote the length of an $R$-module $M$.

\begin{Theorem}(\cite[Theorem 1.1]{C2} and  \cite[Theorem 4.2]{C3})\label{TheoremI20}  Suppose that $R$ is a Noetherian local ring of dimension $d$, and  $N(\hat R)$ is the nilradical of the $m_R$-adic completion $\hat R$ of $R$.  Then   the limit 
\begin{equation}\label{I5}
\lim_{n\rightarrow\infty}\frac{\lambda(R/I_n)}{n^d}
\end{equation}
exists for any filtration  $\mathcal I=\{I_n\}$ of $R$ by $m_R$-primary ideals, if and only if $\dim N(\hat R)<d$.
\end{Theorem}

The nilradical $N(R)$ of a $d$-dimensional ring $R$ is 
$$
N(R)=\{x\in R\mid x^n=0 \mbox{ for some positive integer $n$}\}.
$$
We have  that the dimension of the $R$-module $N(R)$ is $\dim N(R)=d$ if and only if there exists a minimal prime $P$ of $R$ such that $\dim R/P =d$ and $R_P$ is not reduced.

The problem of existence of such limits (\ref{I5}) has also been considered by Ein, Lazarsfeld and Smith \cite{ELS} and Musta\c{t}\u{a} \cite{Mus}.
In the case when the ring $R$ is a domain and is essentially of finite type over an algebraically closed field $k$ with $R/m_R=k$, Lazarsfeld and Musta\c{t}\u{a} \cite{LM} showed that
the limit exists for all filtrations  of $R$ by $m_R$-primary ideals . All of these assumptions are necessary in their proof.  

The following is a very simple example of a filtration of $m_R$-primary ideals such that the above limit is not rational. Let $k$ be a field and $R=k[[x]]$ be a power series ring over $k$. Let $I_n =(x^{\lceil n\sqrt{2}\rceil})$   
where $\lceil \alpha\rceil$ is the round up of a real number $\alpha$ (the smallest integer which is greater than or equal to $\alpha$).  Then $\{I_n\}$ is a graded family of $m_R$-primary ideals such that 
$$
\lim_{n\rightarrow\infty}\frac{\lambda(R/I_n)}{n} =\sqrt{2}
$$ 
is an irrational number. 

There are also irrational examples determined by the valuation ideals of a  discrete valuation. In Example 6 of \cite{CS} an example is given of a normal 3 dimensional  local ring $R$ which is essentially of finite type over a field of arbitrary characteristic and a divisorial valuation $\nu$ on the quotient field of $R$ which dominates $R$ such that the filtration of $m_R$-primary ideals $\{I_n\}$
defined by 
$$
I_n=\{f\in R\mid \nu(f)\ge n\}
$$
satisfies that the limit 
$$
\lim_{n\rightarrow\infty}\frac{\lambda(R/I_n)}{n^3} 
$$ 
is irrational.

Non-Noetherian filtrations ($\oplus_{n\ge 0}I_n$ not Noetherian) occur naturally in commutative algebra. The filtration of ideals determined by a divisorial valuation which dominates a normal local ring is generally not Noetherian. For instance, the condition that a two dimensional normal local ring $R$ satisfies the condition that this filtration is Noetherian for all  divisorial valuations dominating $R$ is the condition (N) of Muhly and Sakuma \cite{MS}. It is proven in \cite{C4} that a complete normal local ring of dimension two satisfies condition (N) if and only if its divisor class group is a torsion group.

The existence of mixed multiplicities of (not necessarily Noetherian) filtrations $\mathcal I(1)=\{I(1)_i\},\ldots,\mathcal I(r)=\{I(r)_i\}$ of $m_R$-primary ideals is established   in Theorem \ref{Theorem2} of this paper. Let $M$ be a finitely generated $R$-module. In Theorem \ref{Theorem2} and Theorem \ref{Cor2}, it is shown that the function
\begin{equation}\label{M2}
P(n_1,\ldots,n_r)=\lim_{m\rightarrow \infty}\frac{\lambda(M/I(1)_{mn_1}\cdots I(r)_{mn_r}M)}{m^d}
\end{equation}
is equal to a homogeneous polynomial $G(n_1,\ldots,n_r)$ of total degree $d$ with real coefficients for all  $n_1,\ldots,n_r\in\NN$.  This limit always exists if and only if the dimension of the nilradial $N(\hat R)$ of the $m_R$-adic completion of $R$ is less than $d=\dim R$, as follows from Theorem \ref{TheoremI20} stated above. We must thus impose the condition that $\dim N(\hat R)<d$. This condition holds if $R$ is analytically unramified; that is, $\hat R$ is reduced.
We may then define the mixed multiplicities of $M$ from the coefficients of $G$, generalizing the definition of mixed multiplicities for $m_R$-primary ideals. Specifically,   
 we write 
$$
G(n_1,\ldots,n_r)=\sum_{d_1+\cdots +d_r=d}\frac{1}{d_1!\cdots d_r!}e_R(\mathcal I(1)^{[d_1]},\ldots, \mathcal I(r)^{[d_r]};M)n_1^{d_1}\cdots n_r^{d_r}.
$$
We say that $e_R(\mathcal I(1)^{[d_1]},\ldots,\mathcal I(r)^{[d_r]};M)$ is the mixed multiplicity of $M$ of type $(d_1,\ldots,d_r)$ with respect to the filtrations $\mathcal I(1),\ldots,\mathcal I(r)$.
Here we are using the notation 
$$
e_R(\mathcal I(1)^{[d_1]},\ldots, \mathcal I(r)^{[d_r]};M)
$$
 for the coefficients of $G(n_1,\ldots,n_r)$ to be consistent with the classical notation for mixed multiplicities of $M$ for $m_R$-primary ideals from \cite{T1}. The mixed multiplicity of $M$ of type $(d_1,\ldots,d_r)$ with respect to $m_R$-primary ideals $I_1,\ldots,I_r$, denoted by $e_R(I_1^{[d_1]},\ldots,I_r^{[d_r]};M)$ (\cite{T1}, \cite[Definition 17.4.3]{HS}) is equal to the mixed multiplicity $e_R(\mathcal I(1)^{[d_1]},\ldots,\mathcal I(r)^{[d_r]};M)$, where the filtrations $\mathcal I(1),\ldots,\mathcal I(r)$ are defined by $\mathcal I(1)=\{I_1^i\}_{i\in \NN}, \ldots,\mathcal I(r)=\{I_r^i\}_{i\in \NN}$.

We write the multiplicity $e_R(\mathcal I;M)=e_R(\mathcal I^{[d]};M)$ if $r=1$, and $\mathcal I=\{I_i\}$ is a filtration of  $R$ by $m_R$-primary ideals. We have that
$$
e_R(\mathcal I;M)=\lim_{m\rightarrow \infty}d!\frac{\lambda(M/I_mM)}{m^d}.
$$
We have by Proposition \ref{Meq}, that  for $1\le i\le r$,
$$
e_R(\mathcal I(i);M)=e_R(\mathcal I(1)^{[0]},\ldots,\mathcal I(i-1)^{[0]},\mathcal I(i)^{[d]},\mathcal I(i+1)^{[0]},\ldots,\mathcal I(r)^{[0]};M),
$$
generalizing the equality for $m_R$-primary ideals by Rees in  \cite[Lemma 2.4]{R}.

We show that many of the classical properties of mixed multiplicities for $m_R$-primary ideals continue to hold for filtrations, including the famous ``Minkowski inequalities'', proven in Theorem \ref{Theorem12}, and stated below. The Minkowski inequalities were formulated and proven for  $m_R$-primary ideals by Teissier \cite{T1}, \cite{T2} and proven in full generality, for Noetherian local rings,  by Rees and Sharp \cite{RS}.
We prove the strong inequality 1) from which the inequalities 2), 3) and 4) follow.   The fourth inequality  4)  was proven for filtrations  of $R$ by $m_R$-primary ideals in a regular local ring with algebraically closed residue field by Musta\c{t}\u{a} (\cite[Corollary 1.9]{Mus}) and more recently by Kaveh and Khovanskii (\cite[Corollary 7.14]{KK1}). The inequality 4) was proven with our assumption that $\dim N(\hat R)<d$ in \cite[Theorem 3.1]{C3}.
Inequalities 2) - 4) can be deduced directly from inequality 1), as explained in \cite{T1}, \cite{T2}, \cite{RS} and \cite[Corollary 17.7.3]{HS}  .

\begin{Theorem}(Minkowski Inequalities)  Suppose that $R$ is a Noetherian $d$-dimensional  local ring with $\dim N(\hat R)<d$, $M$ is a finitely generated $R$-module and $\mathcal I(1)=\{I(1)_j\}$ and $\mathcal I(2)=\{I(2)_j\}$ are filtrations of $R$ by $m_R$-primary ideals. Then 
\begin{enumerate}
\item[1)] $e_R(\mathcal I(1)^{[i]},\mathcal I(2)^{[d-i]};M)^2\le e_R(\mathcal I(1)^{[i+1]},\mathcal I(2)^{[d-i-1]};M)e_R(\mathcal I(1)^{[i-1]},\mathcal I(2)^{[d-i+1]};M)$ 

for $1\le i\le d-1$.
\item[2)]  For $0\le i\le d$, 
$$
e_R(\mathcal I(1)^{[i]},\mathcal I(2)^{[d-i]};M)e_R(\mathcal I(1)^{[d-i]},\mathcal I(2)^{[i]};M)\le e_R(\mathcal I(1);M)e_R(\mathcal I(2);M),
$$
\item[3)] For $0\le i\le d$, $e_R(\mathcal I(1)^{[d-i]},\mathcal I(2)^{[i]};M)^d\le e_R(\mathcal I(1);M)^{d-i}e_R(\mathcal I(2);M)^i$ and
\item[4)]  $e_R(\mathcal I(1)\mathcal I(2));M)^{\frac{1}{d}}\le e_R(\mathcal I(1);M)^{\frac{1}{d}}+e_R(\mathcal I(2);M)^{\frac{1}{d}}$, 

where $\mathcal I(1)\mathcal I(2)=\{I(1)_jI(2)_j\}$.
\end{enumerate}
\end{Theorem}

 In Section \ref{Multi}, we  give an example showing that  Theorem \ref{Theorem2} does not have a good extension to arbitrary multigraded non Noetherian filtrations $\mathcal I=\{I_{n_1,\ldots,n_r}\}$ of $m_R$-primary ideals, even in a power series ring in one variable over a field. In our example ($d=1$)
 $$
 P(n_1,n_2)=\lim_{m\rightarrow \infty}\frac{\lambda(R/I_{mn_1,mn_2})}{m}=\lceil \sqrt{n_1^2+n_2^2}\rceil
$$ 
 for $n_1,n_2\in \NN$, where $\lceil x\rceil$ is the round up of a real number $x$. The function $P(n_1,n_2)$ is  far from polynomial like. 
 
 We will   show however that the function $P(n_1,\ldots,n_r)$ is polynomial like in an important situation. 
 We  show  that the multigraded filtration  of $m_R$-primary ideals   measuring vanishing along the  exceptional divisors of a resolution of singularities of an excellent, normal, two dimensional local ring is such that 
  the function
 $$
 P(n_1,\ldots,n_r)=\lim_{m\rightarrow \infty}\frac{\lambda(R/I_{mn_1,\ldots,mn_r})}{m^2}
$$ 
is  a piecewise polynomial  function (a polynomial with rational coefficients when restricted to  a member of an abstract complex of polyhedral sets whose union is $\QQ_{\ge 0}$). 
The function $P(n_1,\ldots,n_r)$ is in fact  an intersection product on the resolution of singularities.
These formulas hold, even though the  filtration $\{I_{n_1,\ldots,n_r}\}$ is generally not Noetherian.

The first step in the construction of mixed multiplicities for $m_R$-primary filtrations is to construct them for Noetherian filtrations.  In this case the associated multigraded Hilbert function  is a quasi polynomial, whose highest degree terms are constant, rational numbers, as we show in Proposition \ref{Prop3}.  We next restrict in Section \ref{SecVol} to the case $M=R$ and assume that $R$ is analytically irreducible. Using methods of volumes of Newton-Okounkov bodies adapted to our situation, we show in Proposition \ref{Prop2} and Corollary \ref{Cor1} that the coefficients of the polynomials $P_a(n_1,\ldots,n_r)$ of (\ref{M2}) for successive Noetherian approximations $\mathcal I_a(1),\ldots,\mathcal I_a(r)$ of $\mathcal I(1),\ldots, \mathcal I(r)$, all have a limit as $a\rightarrow\infty$. We then define $G(x_1,\ldots,x_n)$ to be the real polynomial with these limit coefficients, and show in Theorem \ref{Theorem1} that for $n_1,\ldots,n_r\in \ZZ_+$, $G(n_1,\ldots,n_r)$ is the function $P(n_1,\ldots,n_r)$ of (\ref{M2}) for the filtrations $\mathcal I(1),\ldots,\mathcal I(r)$.
In Section \ref{Reduc}, we obtain the reductions necessary to prove Theorem \ref{Theorem2}, allowing us to define mixed multiplicities for filtrations of $m_R$-primary ideals in Section \ref{MMult}.

We will denote the nonnegative integers by $\NN$ and the positive integers by $\ZZ_+$.  We will denote the set of nonnegative rational numbers  by $\QQ_{\ge 0}$  and the positive rational numbers by $\QQ_+$. 
We will denote the set of nonnegative real numbers by $\RR_{\ge0}$.

The maximal ideal of a local ring $R$ will be denoted by $m_R$. The quotient field of a domain $R$ will be denoted by ${\rm Q}(R)$. We will denote the length of an $R$-module $M$ by $\lambda_R(M)$ or $\lambda(M)$ if the ring $R$ is clear from context.

A filtration $\mathcal I=\{I_n\}_{n\in\NN}$ of ideals on a ring $R$ is a descending chain
$$
R=I_0\supset I_1\supset I_2\supset \cdots
$$
of ideals such that $I_iI_j\subset I_{i+j}$ for all $i,j\in \NN$. A filtration $\mathcal I=\{I_n\}$ of ideals on a local ring $R$ is a filtration of $R$ by $m_R$-primary ideals if $I_j$ is $m_R$-primary for $j\ge 1$.
A filtration $\mathcal I=\{I_n\}_{n\in\NN}$ of ideals on a ring $R$ is said to be Noetherian if $\bigoplus_{n\ge 0}I_n$ is a finitely generated $R$-algebra.

\section{Polynomials,  Quasi Polynomials and Multiplicities I}\label{Sec2}
A map $\sigma:\NN^r\rightarrow \QQ$ is said to be periodic  if there exists $\alpha\in \NN$ such that
$$
\sigma(n_1,n_2,\ldots,n_i+\alpha,\ldots,n_r)=\sigma(n_1,n_2,\ldots,n_i,\ldots,n_r)
$$ 
for all $(n_1,\ldots,n_r)\in \NN^r$ and $1\le i\le r$. If this condition holds, then $\alpha$ is said to be a period of $\sigma$.

In this section we suppose that  $(R, m_R)$ is a Noetherian local ring, $M$ is a finitely generated $R$-module and $\{I(j)_i\}$ are Noetherian filtrations of $R$ by $ m_R$-primary ideals for all $1\leq j\leq r.$ Then for all $1\leq j\leq r,$ there exists an integer $\alpha\geq 1$ such that
$R_j^{(\alpha)}=\bigoplus\limits_{n\geq 0} I(j)_{\alpha n}$ are Noetherian standard $\mathbb N$-graded rings (by \cite{Bou}[Proposition 3, Section 1.3, Chapter III]). Therefore $$S=\bigoplus\limits_{n_1,\ldots,n_r\geq 0}I(1)_{\alpha n_1}\cdots I(r)_{\alpha n_r}$$ is a Noetherian standard $\mathbb N^r$-graded ring where $S_{(n_1,\ldots,n_r)}=I(1)_{\alpha n_1}\cdots I(r)_{\alpha n_r}.$ For all $1\leq j\leq r,$ consider the ideals \begin{center} $K_j=\bigoplus\limits_{n_1,\ldots,n_r\geq 0}I(1)_{\alpha n_1}\cdots I(j-1)_{\alpha n_{j-1}}I(j)_{\alpha n_j+1}I(j+1)_{\alpha n_{j+1}}\cdots I(r)_{\alpha n_r}$\end{center} of $S$ where $(K_j) _{(n_1,\ldots,n_r)}=I(1)_{\alpha n_1}\cdots I(j-1)_{\alpha n_{j-1}}I(j)_{\alpha n_j+1}I(j+1)_{\alpha n_{j+1}}\cdots I(r)_{\alpha n_r}.$ Then for all $1\leq j\leq r,$
\begin{eqnarray*} G_j^{(\alpha )}&:=&S/K_j\\&=&\bigoplus\limits_{n_1,\ldots,n_r\geq 0}\frac{I(1)_{\alpha n_1}\cdots I(j)_{\alpha n_j}\cdots I(r)_{\alpha n_r}}{I(1)_{\alpha n_1}\cdots I(j-1)_{\alpha n_{j-1}}I(j)_{\alpha n_j+1}I(j+1)_{\alpha n_{j+1}}\cdots I(r)_{\alpha n_r}}\end{eqnarray*} are standard graded algebras over $R/I(j)_1.$
\\For all $1\leq j\leq r$ and integers $0\leq b_j\leq \alpha -1,$ we have finitely generated $G_j^{(\alpha )}$-modules

$$
 \begin{array}{l}G_j^{(b_1,\ldots,b_r)}(M)\\
 \,\,\,\,\,\,\,\,\,\,:=\bigoplus\limits_{n_1,\ldots,n_r\geq 0}\frac{I(1)_{\alpha n_1+b_1}\cdots I(j)_{\alpha n_j+b_j}\cdots I(r)_{\alpha n_r+b_r}M}{I(1)_{\alpha n_1+b_1}\cdots I(j-1)_{\alpha n_{j-1}+b_{j-1}}I(j)_{\alpha n_j+b_j+1}I(j+1)_{\alpha n_{j+1}+b_{j+1}}\cdots I(r)_{\alpha n_r+b_r}M}.
 \end{array}
 $$
By \cite{HHRT}[Theorem 4.1], for all $1\leq j\leq r$ and integers $0\leq b_j\leq \alpha-1,$ there exist polynomials $P_{(b_1,\ldots,b_r)}^{(j)}(X_1,\ldots,X_r)\in\mathbb Q[X_1,\ldots,X_r]$ and an integer $m\in\mathbb Z_+$ such that 
for all $n_1,\ldots,n_r\geq m,$ we have 
$$
\begin{array}{l} H_{(b_1,\ldots,b_r)}^{(j)}(n_1,\ldots,n_r)\\
:=\lambda {\Big(\frac{I(1)_{\alpha n_1+b_1}\cdots I(j)_{\alpha n_j+b_j}\cdots I(r)_{\alpha n_r+b_r}M}{I(1)_{\alpha n_1+b_1}\cdots I(j-1)_{\alpha n_{j-1}+b_{j-1}}I(j)_{\alpha n_j+b_j+1}I(j+1)_{\alpha n_{j+1}+b_{j+1}}\cdots I(r)_{\alpha n_r+b_r}M}\Big)}\\
=P_{(b_1,\ldots,b_r)}^{(j)}(n_1,\ldots,n_r).
\end{array}
$$

\begin{Proposition}{\label{poly}}
	Let $Q_1(X_1,\ldots,X_r),\ldots,Q_k(X_1,\ldots,X_r)\in {\mathbb Q} [X_1,\ldots,X_r]$ be numerical polynomials and $1\leq l$	be a fixed integer. Then for any integer $t\geq 1$ and $j\in\{1,\ldots,r\},$ 
	\begin{center}$\sum\limits_{n=0}^{t}{\sum\limits_{m=1}^{k}Q_m(n_1,\ldots,n_{j-1},l+n,n_{j+1},\ldots,n_r)}$\end{center} is a polynomial $Q(n_1,\ldots,n_{j-1},t,n_{j+1},\ldots,n_r)$ in $n_1,\ldots,n_{j-1},t,n_{j+1},\ldots,n_r$ with coefficients in $\mathbb Q.$
\end{Proposition}
\begin{proof}
	Fix $j.$ For all $m\in\{1,\ldots,k\},$ we have 
	$$
	Q_m(n_1,\ldots,n_r)=\displaystyle\sum\limits_{\substack{\beta=({\beta}_1,\ldots{\beta}_r)\in{\mathbb N}^r \\ |\beta|\leq d_m}}{e_{\beta}^m}\binom{n_1+{{\beta}_1}}{{\beta}_1}\cdots\binom{{n_r}+{{\beta}_r}}{{\beta}_r}$$
	where $d_m$ is the total degree of $Q_m.$ Then $$\tilde{Q}(n_1,\ldots,n_r)=\sum\limits_{m=1}^{k}Q_m(n_1,\ldots,n_r)=\displaystyle\sum\limits_{\substack{\beta=({\beta}_1,\ldots{\beta}_r)\in{\mathbb N}^r \\ |\beta|\leq d}}{u_{\beta}}\binom{n_1+{{\beta}_1}}{{\beta}_1}\cdots\binom{{n_r}+{{\beta}_r}}{{\beta}_r}$$ is a numerical polynomial of total degree less than or equal to $d$ with $d=\max\{d_1,\ldots,d_k\}$ and $u_{\beta}=\sum\limits_{m=1} ^ke_{\beta}^m$ for all $\beta=({\beta}_1,\ldots{\beta}_r)\in{\mathbb N}^r$ with $|\beta|\leq d$ (note that $e_{\beta}^m=0$ if $|\beta|>d_m$). 
\\Let $A_{\beta}:={u_{\beta}}\binom{n_1+{{\beta}_1}}{{\beta}_1}\cdots\binom{n_{j-1}+{{\beta}_{j-1}}}{{\beta}_{j-1}}\binom{n_{j+1}+{{\beta}_{j+1}}}{{\beta}_{j+1}} \cdots\binom{{n_r}+{{\beta}_r}}{{\beta}_r}.$ Then for any integer $t\geq 1,$ 
	
	$$
	\begin{array}{l}\sum\limits_{n=0}^{t}{\sum\limits_{m=1}^{k}Q_m(n_1,\ldots,n_{j-1},l+n,n_{j+1},\ldots,n_r)}\\
	= \sum\limits_{n=0}^{t}\tilde{Q}(n_1,\ldots,n_{j-1},l+n,n_{j+1},\ldots,n_r)\\
	=\sum\limits_{n=0}^{t}{\displaystyle\sum\limits_{\substack{\beta=({\beta}_1,\ldots{\beta}_r)\in{\mathbb N}^r \\ |\beta|\leq d}}}A_{\beta}\binom{l+n+{{\beta}_j}}{{\beta}_j}
		\\
		= \displaystyle\sum\limits_{\substack{\beta=({\beta}_1,\ldots{\beta}_r)\in{\mathbb N}^r \\ |\beta|\leq d}}A_{\beta}\Big[\sum\limits_{n=0}^{t}\binom{l+n+{{\beta}_j}}{{\beta}_j}\Big]
		\\
		=\displaystyle\sum\limits_{\substack{\beta=({\beta}_1,\ldots{\beta}_r)\in{\mathbb N}^r\\ |\beta|\leq d}}A_{\beta}\big[\binom{t+l+{\beta}_j+1}{{\beta}_j+1}-\binom{l+{\beta}_j}{{\beta}_j+1}\big]\\=Q(n_1,\ldots,n_{j-1},t,n_{j+1},\ldots,n_r).\end{array}
		$$
		
\end{proof}	
For all 
$1\leq j\leq r$ and integers $0\leq b_j\leq \alpha -1$,
we define 

$$
(\alpha; b_1,\ldots,b_r)=\{(n_1,\ldots,n_r)\in\mathbb N^r: n_j\equiv b_j(mod\mbox{ } \alpha )\mbox{ for all }1\leq j\leq r\}.
$$
\begin{Proposition}\label{Prop3'}
	Suppose that $R$ is a  Noetherian  local ring, $M$ is a finitely generated $R$-module and $\mathcal I(1)=\{I(1)_i\},\ldots,\mathcal I(r)=\{I(r)_i\}$ are Noetherian  filtrations of $R$ by $m_R$-primary ideals. Consider the function 
	$
	\lambda(M/I(1)_{n_1}\cdots I(r)_{n_r}M)
	$
	of $n_1,\ldots,n_r\in\NN$ where $\lambda$ is the length as an $R$-module.  
	Then  there exist  $c\in \ZZ_+$, $s\in \NN$   and periodic functions $\sigma_{i_1,\ldots,i_r}(n_1,\ldots,n_r)$ such that whenever $n_1,\ldots,n_r\ge c$, we have that
	$$
	\lambda(M/I(1)_{n_1}\cdots I(r)_{n_r}M)=\sum_{i_1+\cdots+i_r\le s}\sigma_{i_1,\ldots,i_r}(n_1,\ldots,n_r)n_1^{i_1}n_2^{i_2}\cdots n_r^{i_r}.
	$$
\end{Proposition}

\begin{proof}
	(We use the integer $\alpha$ and the polynomials $P_{(b_1,\ldots,b_{r})}^{(j)}$  mentioned in the above discussion.)
	\\For all $1\leq j\leq r$ and integers $0\leq b_j\leq \alpha -1,$  we define the polynomials 
	$$
	 \begin{array}{l}
	 Q_{(b_1,\ldots,b_j,0,\ldots,0)}^{(j)}(n_1,\ldots,n_r)\\
	  = \left\{
	\begin{array}{l l}
	0  & \quad \text{if $b_j=0$ }\\
	\sum\limits_{i(j)=1}^{b_j}P_{(b_1,\ldots,b_{j-1},i(j)-1,0,\ldots,0)}^{(j)}(n_1,\ldots,n_r) & \quad \text{if $1\leq b_j\leq \alpha -1.$ }
	\end{array} \right.
	\end{array}
	$$

	Let $\alpha \leq t=\alpha l \in \mathbb N$ be such that $H_{(b_1,\ldots,b_r)}^{(j)}(m_1,\ldots,m_r)=P_{(b_1,\ldots,b_r)}^{(j)}(m_1,\ldots,m_r)$ for all $m_1,\ldots,m_r\geq l$ with $0\leq b_1,\ldots,b_r\leq \alpha -1.$ Let $(n_1,\ldots,n_r)\in (\alpha ;b_1,\ldots,b_r)$ with $n_j\geq c=t+\alpha $ for all $1\leq j\leq r.$ Then
	
	$$
	\begin{array}{l}
		\lambda(M/I(1)_{n_1}\cdots I(r)_{n_r}M)\\
		= \lambda(M/I(1)_{t}\cdots I(r)_{t}M)+\sum\limits_{i=0}^{n_1-t-1}\lambda\Big(\frac{I(1)_{t+i}I(2)_{t}\cdots I(r)_{t}M}{I(1)_{t+i+1}I(2)_{t}\cdots I(r)_{t}M}\Big)\\
		+ \sum\limits_{i=0}^{n_2-t-1}\lambda\Big(\frac{I(1)_{n_1}I(2)_{t+i}I(3)_{t}\cdots I(r)_{t}M}{I(1)_{n_1}I(2)_{t+i+1}I(3)_{t}\cdots I(r)_{t}M}\Big)\\
		+\cdots + \sum\limits_{i=0}^{n_r-t-1}\lambda\Big(\frac{I(1)_{n_1}I(2)_{n_2}\cdots I(r-1)_{n_{r-1}}I(r)_{t+i}M}{I(1)_{n_1}I(2)_{n_2}\cdots I(r-1)_{n_{r-1}}I(r)_{t+i+1}M}\Big)
		\\
		= \lambda(M/I(1)_{t}\cdots I(r)_{t}M)\\
		+{\sum\limits_{p(1)=0}^{\lfloor{\frac{n_1-t}{\alpha }}\rfloor-1}}{\sum\limits_{i(1)=1}^{\alpha }}P_{(i(1)-1,0,\ldots,0)}^{(1)}(l+p(1),l,\ldots,l)+ Q_{(b_1,0,\ldots,0)}^{(1)}(l+{\lfloor{\frac{n_1-t}{\alpha }}\rfloor},l,\ldots,l)
		\\
		+{\sum\limits_{p(2)=0}^{\lfloor{\frac{n_2-t}{\alpha }}\rfloor-1}}{\sum\limits_{i(2)=1}^{\alpha }}P_{(b_1,i(2)-1,0,\ldots,0)}^{(2)}({l+\lfloor{\frac{n_1-t}{\alpha }}\rfloor},l+p(2),l,\ldots,l)\\
		+ Q_{(b_1,b_2,0,\ldots,0)}^{(2)}(l+{\lfloor{\frac{n_1-t}{\alpha }}\rfloor},l+{\lfloor{\frac{n_2-t}{\alpha }}\rfloor},l,\ldots,l)\\
		\,\,\,\,\,\vdots\\
		+{\sum\limits_{p(r)=0}^{\lfloor{\frac{n_r-t}{\alpha }}\rfloor-1}}{\sum\limits_{i(r)=1}^{\alpha }}P_{(b_1,\ldots,b_{r-1},i(r)-1)}^{(r)}(l+{\lfloor{\frac{n_1-t}{\alpha }}\rfloor},\ldots,l+{\lfloor{\frac{n_{r-1}-t}{\alpha }}\rfloor},l+p(r))\\
		+ Q_{(b_1,\dots,b_{r})}^{(r)}(l+{\lfloor{\frac{n_1-t}{\alpha }}\rfloor},\ldots,l+{\lfloor{\frac{n_r-t}{\alpha }}\rfloor})
		\\
		= \lambda(M/I(1)_{t}\cdots I(r)_{t}M)\\
		+ {\sum\limits_{j=1}^{r}}\Big[{\sum\limits_{p(j)=0}^{\lfloor{\frac{n_j-t}{\alpha }}\rfloor-1}}{\sum\limits_{i(j)=1}^{\alpha }}P_{(b_1,\ldots,b_{j-1},i(j)-1,0,\ldots,0)}^{(j)}(l+{\lfloor{\frac{n_1-t}{\alpha }}\rfloor},\ldots,l+{\lfloor{\frac{n_{j-1}-t}{\alpha }}\rfloor},l+p(j),l,\ldots,l)\\
		+ Q_{(b_1,\ldots,b_j,0,\ldots,0)}^{(j)}(l+{\lfloor{\frac{n_1-t}{\alpha }}\rfloor},\ldots,l+{\lfloor{\frac{n_{j}-t}{\alpha }}\rfloor},l,\ldots,l)\Big]
	\end{array}
	$$	
	Using Proposition \ref{poly}, we have  a multigraded polynomial $$T_{(b_1,\ldots,b_r)}(X_1,\ldots,X_r):=\sum\limits_{i_1+\cdots+i_r\leq u{(b_1,\ldots,b_r)}}e_{(i_1,\ldots,i_r)}^{(b_1,\ldots,b_r)}X_1^{i_1}\cdots X_r^{i_r}\in\mathbb Q[X_1,\ldots,X_r]$$ 
	such that 
	$$
	\begin{array}{l}T_{(b_1,\ldots,b_r)}(m_1,\ldots,m_r)\\
	=\lambda(M/I(1)_{t}\cdots I(r)_{t}M)\\
	+ {\sum\limits_{j=1}^{r}}\Big[{\sum\limits_{p(j)=0}^{m_j-1}}\Big({\sum\limits_{i(j)=1}^{\alpha }}P_{(b_1,\ldots,b_{j-1},i(j)-1,0,\ldots,0)}^{(j)}(l+m_1,\ldots,l+m_{j-1},l+p(j),l,\ldots,l)\Big)\\
	+ Q_{(b_1,\ldots,b_j,0,\ldots,0)}^{(j)}(l+m_1,\ldots,l+m_j,l,\ldots,l)\Big]	
	\end{array}	
	$$
	and for all $(n_1,\ldots,n_r)\in (\alpha ;b_1,\ldots,b_r)$ with $ n_j\geq c$ for $1\leq j\leq r,$ we get \begin{eqnarray*}
		\lambda(M/I(1)_{n_1}\cdots I(r)_{n_r}M)&=&T_{(b_1,\ldots,b_r)}(a(n_1),\ldots,a(n_r))\end{eqnarray*}	where $a(n_j):={\lfloor{\frac{n_j-t}{\alpha }}\rfloor}=\frac{n_j-t-b_j}{\alpha }$ for all $1\leq j\leq r$ and let $u{(b_1,\ldots,b_r)}$ be the total degree of $T_{(b_1,\ldots,b_r)}.$   
	\\Now for all $(n_1,\ldots,n_r)\in (\alpha ;b_1,\ldots,b_r),$ we have
	$$
	\begin{array}{l}
			T_{(b_1,\ldots,b_r)}(a(n_1),\ldots,a(n_r)) \\
			\,\,\,\,\,=\sum\limits_{i_1+\cdots+i_r\leq u{(b_1,\ldots,b_r)}}e_{(i_1,\ldots,i_r)}^{(b_1,\ldots,b_r)}a(n_1)^{i_1}\cdots a(n_r)^{i_r}\\
			\,\,\,\,\,= \sum\limits_{i_1+\cdots+i_r\leq u{(b_1,\ldots,b_r)}}e_{(i_1,\ldots,i_r)}^{(b_1,\ldots,b_r)} \frac{(n_1-t-b_1)^{i_1}\cdots(n_r-t-b_r)^{i_r}}{\alpha ^{i_1+\cdots+i_r}}.
	\end{array}
	$$
	Let $\sigma_{(i_1,\ldots,i_r)}^{(b_1,\ldots,b_r)}(n_1,\ldots,n_r)$ denote the coefficient of $n_1^{i_1}\cdots n_r^{i_r}$ in the above equation. 
	\\Let $e_j=(0,\ldots,0,1,0,\ldots,0)\in\mathbb N^r$ with $1$ at the $j$-th position where $j\in\{1,\ldots,r\}.$ Note that $$(n_1,\ldots,n_r)+\alpha e_j\in (\alpha ;b_1,\ldots,b_r)\mbox{ and }a(n_j+\alpha )={\lfloor{\frac{n_j+\alpha -t}{\alpha }}\rfloor}=\frac{n_j+\alpha -t-b_j}{\alpha }.$$ Thus for all $1\leq j\leq r,$ we have  $$\sigma_{(i_1,\ldots,i_r)}^{(b_1,\ldots,b_r)}(n_1,\ldots,n_r)=\sigma_{(i_1,\ldots,i_r)}^{(b_1,\ldots,b_r)}(n_1,\ldots,n_{j-1},n_j+\alpha ,n_{j+1}\ldots,n_r).$$ 
	For all $(m_1,\ldots,m_r)\in(\alpha ;b_1,\ldots,b_r),$ we define  \[ \sigma_{i_1,\ldots,i_r}(m_1,\ldots,m_r) = \left\{
	\begin{array}{l l}
	\sigma_{(i_1,\ldots,i_r)}^{(b_1,\ldots,b_r)}(m_1,,\ldots,m_r)  & \quad \text{if $i_1+\cdots+i_r\leq u{(b_1,\ldots,b_r)}$ }\\
	0 & \quad \text{if $i_1+\cdots+i_r> u{(b_1,\ldots,b_r)}.$ }
	\end{array} \right.\] 
	Therefore for all $n_1,\ldots,n_r\geq c$ and $s=\max\{u{(b_1,\ldots,b_r)}: 0\leq b_1,\ldots,b_r\leq \alpha -1\},$ we get $$\lambda(M/I(1)_{n_1}\cdots I(r)_{n_r}M)=\sum\limits_{i_1+\cdots+i_r\leq s}\sigma_{i_1,\ldots,i_r}(n_1,\ldots,n_r)n_1^{i_1}\cdots n_r^{i_r}.$$ 
\end{proof}

\section{Polynomials,  Quasi Polynomials and Multiplicities II}\label{Sec3}

\begin{Lemma}\label{Lemma1}  Suppose that  $r,d\ge 1$ and $a=\binom{r-1+d}{r-1}$.  Then there exist $n_1(i),\ldots n_r(i)\in\ZZ_+$ for $1\le i\le a$ such that the set of  vectors consisting of all monomials of degree $d$ in $n_1(i),\ldots,n_r(i)$ for $1\le i\le a$,
	$$
	\{(n_1(1)^d,n_1(1)^{d-1}n_2(1),\ldots,n_r(1)^d),\ldots,(n_1(a)^d,n_1(a)^{d-1}n_2(a),\ldots,n_r(a)^d)\}
	$$
	is a $\QQ$-basis of $\QQ^a$.
\end{Lemma} 

\begin{proof} Let $\Lambda:(\QQ_+)^r\rightarrow \QQ^a$ be defined by $\Lambda(s_1,\ldots,s_r)=(s_1^d,s_1^{d-1}s_2,\ldots,s_r^d)$. We will first show that the image of $\Lambda$ is not contained in a proper $\QQ$-linear subspace of $\QQ^a$. Suppose otherwise. Then there exists a nonzero linear form 
	$$
	L(y_{d,0,\ldots,0},y_{d-1,1,0,\ldots,0},\ldots,y_{0,\ldots,0,d})=\sum_{i_1+\cdots+i_r=d}\alpha_{i_1,\ldots,i_r}y_{i_1,\ldots,i_r}
	$$
	on $\QQ^a$  such that $L(s_1^d,s_1^{d-1}s_2,\ldots,s_r^d)=0$ for all $(s_1,\ldots,s_r)\in (\QQ_+)^r$. The degree $d$ form $G(x_1,\ldots,x_r):= L(x_1^d,x_1^{d-1}x_2,\ldots,x_r^d)$ vanishes on $(\QQ_+)^r$. Since $\QQ$ is an infinite field, this implies that $G(x_1,\ldots,x_r)$ is the zero polynomial (as follows from the proof of Theorem 2.19 \cite{J}).  But $G(x_1,\ldots,x_r)$ is a nontrivial linear combination of the monomials in $x_1,\ldots, x_r$ of degree $d$, so it cannot be zero.  So $\mbox{Image}(\Lambda)$ is not contained in a proper linear subspace of $\QQ^a$. Thus there exist
	$(s_1(i),\ldots,s_r(i))\in (\QQ_+)^r$ for $1\le i\le a$ such that 
	$$
	\{(s_1(1)^d,s_1(1)^{d-1}s_2(1),\ldots,s_r(1)^d),\ldots,(s_1(a)^d,s_1(a)^{d-1}s_2(a),\ldots,s_r(a)^d)\}
	$$
	is a $\QQ$-basis of $\QQ^a$. There exists a positive integer $u$ such that $n_i(j)=us_i(j)\in \ZZ_+$ for all $i,j$, and since
	$$
	(n_1(j)^d,n_1(j)^{d-1}n_2(j),\ldots,n_r(j)^d)=u^d(s_1(j)^d,s_1(j)^{d-1}s_2(j),\ldots,s_r(j)^d)
	$$
	for $1\le j\le a$, we have that 
	$$
	\{(n_1(1)^d,n_1(1)^{d-1}n_2(1),\ldots,n_r(1)^d),\ldots,(n_1(a)^d,n_1(a)^{d-1}n_2(a),\ldots,n_r(a)^d)\}
	$$
	is a $\QQ$-basis of $\QQ^d$.
\end{proof}

\begin{Lemma}\label{Prop1}   Let $g=\binom{r-1+d}{r-1}$. There exist $n_1(i),\ldots,n_r(i)\in \ZZ_+$ for  $1\le i\le g$ and $c_j(i_1,\ldots,i_r)\in \QQ$ for $1\le j\le g$ and  $i_1,\ldots,i_r\in \NN$ with $i_1+\cdots+i_r=d$, such that if  $F(x_1,\ldots,x_r)\in \QQ[x_1,\ldots,x_r]$ is a polynomial of total degree $d$, with an expansion
	\begin{equation}\label{eq1}
	F(x_1,\ldots,x_r)=\sum_{i_1+\cdots+i_r\le d}a_{i_1,\ldots,i_r}x_1^{i_1}\cdots x_r^{i_r}\in \QQ[x_1,\ldots,x_r]
	\end{equation}
	with $a_{i_1,\ldots,i_r}\in \QQ$, then
	for $i_1,\ldots,i_r\in \NN$ with $i_1+\cdots+i_r=d$,
	\begin{equation}\label{eq3}
	a_{i_1,\ldots,i_r}=\sum_{j=1}^g c_j(i_1,\ldots,i_r)b_j
	\end{equation}
	where
	\begin{equation}\label{eq2}
	b_j=\lim_{m\rightarrow \infty}\frac{F(mn_1(j),\ldots,mn_r(j))}{m^d}.
	\end{equation}
\end{Lemma}

\begin{proof} By Lemma \ref{Lemma1}, we can choose $n_i(j)\in \ZZ_+$ for $1\le i\le r$ and $1\le j\le g$ so that
	$$
	B=\left(\begin{array}{cccc}
	n_1(1)^d&n_1(1)^{d-1}n_2(1)&\cdots&n_r(1)^d\\
	\vdots&&&\vdots\\
	n_1(g)^d&n_1(g)^{d-1}n_2(g)&\cdots&n_r(g)^d
	\end{array}\right)
	$$
	has rank $g$. Write
	\begin{equation}\label{eq6}
	B^{-1}=\left(\begin{array}{ccc}
	c_1(d,0,\ldots,0)&\cdots&c_g(d,0,\ldots,0)\\
	c_1(d-1,1,0,\ldots,0)&\cdots&c_g(d-1,1,0,\ldots,0)\\
	\vdots&\vdots\\
	c_1(0,\ldots,0,d)&\cdots&c_g(0,\ldots,0,d)
	\end{array}\right).
	\end{equation}
	Suppose $F(x_1,\ldots,x_r)\in \QQ[x_1,\ldots,x_r]$ has the expression (\ref{eq1}). By (\ref{eq1}) and (\ref{eq2}),
	$$
	b_j=\lim_{t\rightarrow\infty}\frac{F(tn_1(j),\ldots,tn_r(j))}{t^d}
	=\sum_{i_1+\cdots+i_r=d}a_{i_1,\ldots,i_r}n_1(j)^{i_1}\cdots n_r(j)^{i_r}
	$$
	for $1\le j\le g$. We thus have that 
	$$
	B\left(\begin{array}{c} a_{d,0,\ldots,0}\\
	a_{d-1,1,0,\ldots,0}\\
	\vdots\\
	a_{0,\ldots,0,d}\end{array}\right)
	=
	\left(\begin{array}{c}
	b_1\\b_2\\
	\vdots\\b_g\end{array}\right),
	$$
	so that (\ref{eq3}) holds by (\ref{eq6}).
\end{proof}

Suppose that $R$ is a Noetherian local ring of dimension $d,$ $M$ is a finitely generated $R$-module and $J$ is an $m_R$-primary ideal in $R$. Recall that 
the multiplicity $e_R(J;M)$  is defined by the expansion of the Hilbert polynomial of $M$, which is equal to $\lambda(M/J^mM)$ for $m\gg 0$, 
$$
\frac{e_R(J;M)}{d!}m^d+\mbox{ lower order terms in $m$,}
$$
so that 
$$
e_R(J;M)=\lim_{m\rightarrow \infty}d!\frac{\lambda(M/J^mM)}{m^d}.
$$

\begin{Lemma}\label{Lemma0} Suppose $R$ is a Noetherian local ring of dimension $d,$ $M$ is a finitely generated $R$-module  and 
	$$
	\mathcal I(1)=\{I(1)_i\},\ldots,\mathcal I(r)=\{I(r)_i\}
	$$
	are Noetherian filtrations of $R$ by $m_R$-primary ideals. Let $a\in \ZZ_+$ be such that $I(j)_{ia}=I(j)_a^i$ for $1\le j\le r$ and $i\ge 0$. Suppose $n_1,\ldots,n_r\in \NN$. Then 
	$$
	\lim_{m\rightarrow \infty} \frac{\lambda(M/I(1)_{mn_1}\cdots I(r)_{mn_r}M)}{m^d}=\frac{1}{d!a^d}e_R(I(1)_{an_1}\cdots I(r)_{an_r};M)\in \QQ_+.
	$$
\end{Lemma}

\begin{proof} For $m\in \ZZ_+$, write $m=ua+v$ with $0\le v<a$. Then we have a short exact sequence of $R$-modules
	$$
	\begin{array}{l}
	0\rightarrow I(1)_{uan_1}\cdots I(r)_{uan_r}M/I(1)_{mn_1}\cdots I(r)_{mn_r}M\rightarrow M/I(1)_{mn_1}\cdots I(r)_{mn_r}M\\
	\rightarrow
	M/I(1)_{uan_1}\cdots I(r)_{uan_r}M\rightarrow 0.
	\end{array}
	$$
	We have that  for $m\gg 0$,
	$$
	\begin{array}{l}
	\lambda(I(1)_{uan_1}\cdots I(r)_{uan_r}M/I(1)_{mn_1}\cdots I(r)_{mn_r}M)\\
	\le 
	\lambda(M/I(1)_{(u+1)an_1}\cdots I(r)_{(u+1)an_r}M)-\lambda(M/I(1)_{uan_1}\cdots I(r)_{uan_r}M)\\
	=\frac{e_R(I(1)_{an_1}\cdots I(r)_{an_r};M)}{(d-1)!}u^{d-1}+\mbox{ lower order terms in $u$.}
	\end{array}
	$$
	So
	$$
	\begin{array}{l}
	\lim_{m\rightarrow \infty}\frac{\lambda(I(1)_{uan_1}\cdots I(r)_{uan_r}M/I(1)_{mn_1}\cdots I(r)_{mn_r}M)}{m^d}\\
	\le \lim_{u\rightarrow \infty}\frac{\frac{
			e_R(I(1)_{an_1}\cdots I(r)_{an_r};M)
		}{(d-1)!}u^{d-1}+\mbox{ lower order terms in $u$}}{(ua+v)^d}=0.
	\end{array}
	$$
	Thus
	$$
	\begin{array}{lll}
	\lim_{m\rightarrow \infty}\frac{\lambda(M/I(1)_{mn_1}\cdots I(r)_{mn_r}M)}{m^d}&=&
	\lim_{u\rightarrow\infty}\frac{\lambda(M/I(1)_{uan_1}\cdots I(r)_{uan_r}M)}{(ua+v)^d}\\
	&=&\frac{1}{d!a^d}e_R(I(1)_{an_1}\cdots I(r)_{an_r};M).
	\end{array}
	$$
	
\end{proof}

Define the total degree of a quasi polynomial $\sum \sigma_{i_1,\ldots,i_r}(n_1,\ldots,n_r)n_1^{i_1}\cdots n_r^{i_r}$ to be the largest $t$ such that there exists $i_1,\ldots,i_r\in \NN$ with $i_1+\cdots+i_r=t$, such that $\sigma_{i_1,\ldots,i_r}(n_1,\ldots,n_r)$ is not (identically) zero.

\begin{Proposition}\label{Prop4} Let
	$$
	P(n_1,\ldots,n_r)=\sum \sigma_{i_1,\ldots,i_r}(n_1,\ldots,n_r)n_1^{i_1}\cdots n_r^{i_r}
	$$
	be the quasi polynomial of the conclusions of Proposition \ref{Prop3'}. Then the total degree of $P(n_1,\ldots,n_r)$ is $\dim M$, and $\sigma_{i_1,\ldots,i_r}(n_1,\ldots,n_r)$ is a constant function if $i_1+\cdots+i_r=\dim M$.
\end{Proposition}

\begin{proof} Let $t$ be the total degree of $P(n_1,\ldots,n_r)$ and let $a\in \ZZ_+$ be such that $I(j)_{ai}=I(j)_a^i$ for all $i\ge 0$ and $1\le j\le r$, so that $a$ is a common period of the coefficients $\sigma_{i_1,\ldots,i_r}(n_1,\ldots,n_r)$ of $P(n_1,\ldots,n_r)$ (by the proof of Proposition \ref{Prop3'}). Suppose that $b_1,\ldots,b_r\in\NN$ with $0\le b_i<a$ for all $i$. Suppose $n_1,\ldots,n_r\in \ZZ_+$. Then for $n_1,\ldots,n_r\gg 0$,
	$$
	\lambda(M/I(1)_{an_1+b_1}\cdots I(r)_{an_r+b_r}M)=P(an_1+b_1,\ldots,an_r+b_r).
	$$
	Define 
	$$
	\begin{array}{l}
	{P}_{(b_1,\ldots,b_r)}(n_1,\ldots,n_r) := P(an_1+b_1,\ldots,an_r+b_r)\\
	=\sum_{i_1+\cdots+i_r\le t}\sigma_{i_1,\ldots,i_r}(an_1+b_1,\ldots,an_r+b_r)(an_1+b_1)^{i_1}\cdots(an_r+b_r)^{i_r}\\
	=\sum_{i_1+\cdots+i_r\le t}\sigma_{i_1,\ldots,i_r}(b_1,\ldots,b_r)(an_1+b_1)^{i_1}\cdots(an_r+b_r)^{i_r}\\
	=\sum_{i_1+\cdots+i_r=t}\sigma_{i_1,\ldots,i_r}(b_1,\ldots,b_r)a^tn_1^{i_1}\cdots n_r^{i_r}\\
	\,\,\,\,\, +\mbox{ lower total order terms in }n_1,\ldots,n_r.
	\end{array}
	$$
	We have that ${P}_{(b_1,\ldots,b_r)}(n_1,\ldots,n_r)\in \QQ[n_1,\ldots,n_r]$ is a polynomial. For fixed $n_1,\ldots,n_r\in \ZZ_+$ and $m>>0,$ we have $${P}_{(0,\ldots,0)}(mn_1,\ldots,mn_r)=\lambda(M/I(1)_{amn_1}\cdots I(r)_{amn_r}M)=\lambda(M/(I(1)_{an_1}\cdots I(r)_{an_r})^mM).$$ Thus by \cite{HS}[Lemma 11.1.3], 	$$
	\lim_{m\rightarrow \infty}\frac{{P}_{(0,\ldots,0)}(mn_1,\ldots,mn_r)}{m^{dim M}}\in \QQ_+.
		$$
	Therefore the total degree of ${P}_{(0,\ldots,0)}(n_1,\ldots,n_r)$ is $\dim M$.
	
	Fix $n_1,\ldots, n_r\in \ZZ_+$ and $b_i\in \NN$ with $0\le b_i<a$ for $1\le i\le r$. For $m\in \ZZ_+$, we have short exact sequences of $R$-modules,
	$$
	\begin{array}{l}
	0\rightarrow I(1)_{man_1}\cdots I(r)_{man_r}M/I(1)_{man_1+b_1}\cdots I(r)_{man_r+b_r}M\rightarrow M/I(1)_{man_1+b_1}\cdots I(r)_{man_r+b_r}M\\
	\rightarrow M/I(1)_{man_1}\cdots I(r)_{man_r}M\rightarrow 0.
	\end{array}
	$$
	Now for $m\gg 0$,
	$$
	\begin{array}{l}
	\lambda(I(1)_{man_1}\cdots I(r)_{man_r}M/I(1)_{man_1+b_1}\cdots I(r)_{man_r+b_r}M)\\
	\le \lambda(I(1)_{man_1}\cdots I(r)_{man_r}M/I(1)_{(m+1)an_1}\cdots I(r)_{(m+1)an_r}M)\\
	={P}_{(0,\ldots,0)}((m+1)n_1,\ldots,(m+1)n_r)-{P}_{(0,\ldots,0)}(mn_1,\ldots,mn_r)\\
	\end{array}
	$$
	is a polynomial of degree less than $\dim M$ in $m$. Thus
	$$
	\begin{array}{lll}
	\lim_{m\rightarrow \infty}\frac{{P}_{(b_1,\ldots,b_r)}(mn_1,\ldots,mn_r)}{m^{\dim M}}
	&=& \lim_{m\rightarrow \infty}\frac{\lambda(M/I(1)_{man_1}\cdots I(r)_{man_r}M)}{m^{\dim M}}\\
	&=&\lim_{m\rightarrow \infty}\frac{{P}_{(0,\ldots,0)}(mn_1,\ldots,mn_r)}{m^{\dim M}}
	\end{array}
	$$
	and 
	$$
	\sigma_{i_1,\ldots,i_r}(b_1,\ldots,b_r)=\sigma_{i_1,\ldots,i_r}(0,\ldots,0)
	$$
	if $i_1+\cdots+i_r=\dim M$ by Lemma \ref{Prop1}.

\end{proof}

\begin{Proposition}\label{Prop3}  Suppose that $R$ is a  Noetherian local ring, $M$ is a finitely generated $R$-module and $\mathcal I(1)=\{I(1)_i\},\ldots,\mathcal I(r)=\{I(r)_i\}$ are Noetherian filtrations of $R$ by $m_R$-primary ideals. Then  there exist a positive integer $c$  and periodic functions $\sigma_{i_1,\ldots,i_r}(n_1,\ldots,n_r)$ such that whenever $n_1,\ldots,n_r\ge c$, we have that
	$$
	\lambda(M/I(1)_{n_1}\cdots I(r)_{n_r}M)=\sum_{i_1+\cdots+i_r\le \dim M}\sigma_{i_1,\ldots,i_r}(n_1,\ldots,n_r)n_1^{i_1}n_2^{i_2}\cdots n_r^{i_r}
	$$
	is a quasi polynomial of total degree equal to $\dim M$,
	and  the coefficients $\sigma_{i_1,\ldots,i_r}(n_1,\ldots,n_r)$  are constants whenever $i_1+\cdots+i_r=\dim M$.
\end{Proposition}

\begin{proof} This follows from Propositions \ref{Prop3'} and \ref{Prop4}. \end{proof}

\section{Volumes on analytically irreducible local domains}\label{SecVol}

\begin{Definition}\label{trunc}
 Suppose that $\mathcal I=\{I_i\}$  is a filtration of ideals on a local ring $R$. For $a\in \ZZ_+$, the $a$-th truncated  filtration $\mathcal I_a=\{I_{a,i}\}$ of $\mathcal I$ is defined  by $I_{a,n}=I_n$ if $n\le a$ and if $n>a$, then $I_{a,n}=\sum I_{a,i}I_{a,j}$ where the sum is over $i,j>0$ such that $i+j=n$.  
 \end{Definition}

We give an algebraic  proof of the following lemma. A  geometric proof is given on page 9 of \cite{C2}.
\begin{Lemma} \label{Res} Suppose that $R$ is an excellent $d$-dimensional local domain. Then there exists an excellent regular local ring $S$ of dimension $d$ which birationally dominates $R$.
\end{Lemma}

\begin{proof} Let $d=\dim R$. Let $z_1,\ldots,z_d$ be a system of parameters in $R$ and let $Q=(z_1,\ldots,z_d)$, which is an $m_R$-primary ideal in $R$. 
Let $T$ be the integral closure of $B=R[\frac{z_2}{z_1},\ldots,\frac{z_d}{z_1}]$ in ${\rm Q}(R)$. The ring $T$ is an excellent ring and is a finitely generated $R$-algebra by  \cite[Theorem 78, page 257]{Ma2}.

We will now show that $z_1$ is not a unit in $B$, using an argument from \cite[(1.3.1) on page 15]{Ab}. Suppose that $z_1$ is a unit in $B$. Then there exists  $y\in B$ such that $z_1y=1$, so there exists a nonzero polynomial $f(X_2,\ldots,X_d)$ of some degree $n$ with coefficients in $R$ such that $y=f\left(\frac{z_2}{z_1},\ldots,\frac{z_d}{z_1}\right)$. Then $z_1^n=z_1^{n+1}y=z_1g(z_1,\ldots,z_d)$ where $g(X_1,\ldots,X_d)$ is a nonzero homogeneous polynomial of degree $n$ with coefficients in $R$. Thus $z_1^n\in m_RQ^n$, which is a contradiction by \cite[Theorem 21 on page 292]{ZS2}. We further have that $z_1$ is not a unit in $T$ since $T$ is finite over $B$.
Now  $QT =z_1T$  and $z_1$ is not a unit in $T$ and so  ${\rm ht}(P)=1$ if $P$ is a minimal prime of $m_RT$ by Krull's principal ideal theorem.

We next show that $T$ has dimension $d$. The ring $R$ is universally catenary since $R$ is excellent, so the dimension formula holds between $R$ and $T$ (the inequality (*) on page 85 \cite{Ma2} is an equality). 
Let $n$ be a maximal ideal of $T$ which contains $z_1$. Then $n\cap R=m_R$.
We have that $T/n$ is a finitely generated algebra over the field $R/m_R$ and $T/n$ is a field, so  that $T/n$ is a finite $R/m_R$-module by  \cite[Corollary 1.2, page 379]{L}.
By the dimension formula, we have that
$$
{\rm ht}(n)={\rm ht}(m_R)+{\rm trdeg}_{{\rm Q}(R)}{\rm Q}(T)-{\rm trdeg}_{R/m_R}T/nT={\rm ht}(m_R)= d.
$$
Since the dimension formula gives us that ${\rm ht}(m)\le d $ for all maximal ideals $m$ in $T$, we have that
$\dim T=d$.
Let 
$$
{\rm NR}(T)=\{P\in{\rm Spec}(T)\mid T_P\mbox{ is not  a regular local ring}\}.
$$
The set ${\rm NR}(T)$ is a closed set since $T$ is excellent. Let $I$ be an ideal of $T$ such that ${\rm NR}(T)={\rm Spec}(T/I)$. If $P$ is a minimal prime of $I$ then ${\rm ht}(P)>1$ since $T$ is normal (Serre's criterion for normality).  The Jacobson radical of $T/m_RT$ (the intersection of all maximal ideals of $T/m_RT$) is the nilradical of $T/m_RT$ by  \cite[Theorem 25, page 93]{Ma2}, since $T/m_RT$ is a finitely generated algebra over the field $R/m_R$. Let $\overline I=I(T/m_RT)$. There exists a maximal ideal $\overline n$ of $T/m_RT$ such that $\overline I\not\subset \overline n$ since otherwise $I\subset\sqrt{m_RT}$ which is impossible, since all minimal primes of $I$ have height larger than 1 and all minimal primes of $m_RT$ have height equal to one. 
Let $n$ be the lift of $\overline n$ to a maximal ideal of $T$. Then  $S:=T_n$ is a regular local ring of dimension $d$ which birationally dominates $R$.
\end{proof}

In this section, we suppose that $R$ is a Noetherian local ring  of dimension $d$ which is analytically irreducible. Suppose that $\mathcal I(1)=\{I(1)_i\},\ldots,\mathcal I(r)=\{I(r)_i\}$ are (not necessarily Noetherian) filtrations of $R$ by $m_R$-primary ideals.
Define a function $F:\NN^r\rightarrow \RR$ by 
\begin{equation}\label{eq4}
F(n_1,\ldots,n_r)=\lim_{m\rightarrow\infty}\frac{\lambda(R/I(1)_{mn_1}\cdots I(r)_{mn_r})}{m^d}
\end{equation}
for $n_1,\ldots,n_r\in \NN$
where the limit is over $m\in \ZZ_+$. This limit exists by Theorem \ref{TheoremI20}.

For $a\in \ZZ_+$, let $\{I_a(j)_i\}$ be the $a$-th truncated filtration  of $\{I(j)_i\}$ for $1\le j\le r$ (defined in Definition \ref{trunc}). By Proposition \ref{Prop3}, for $a\in \ZZ_+$, there is a homogeneous polynomial $F_a(x_1,\ldots,x_r)$ of total degree $d$ in $\QQ[x_1,\ldots,x_r]$, such that
$$
\lim_{m\rightarrow\infty} \frac{\lambda(R/I_a(1)_{mn_1}\cdots I_a(r)_{mn_r})}{m^d}=F_a(n_1,\ldots,n_r)
$$
if $n_1,\ldots,n_r\in\ZZ_+$. Expand
$$
F_a(x_1,\ldots,x_r)=\sum_{i_1+\cdots+i_r=d}b_{i_1,\ldots,i_r}(a)x_1^{i_1}\cdots x_r^{i_r}
$$
with $b_{i_1,\ldots,i_r}(a)\in \QQ$.

\begin{Proposition}\label{Prop2} For fixed $n_1,\ldots,n_r\in\ZZ_+$,
$$
\lim_{a\rightarrow\infty} F_a(n_1,\ldots,n_r)=F(n_1,\ldots,n_r).
$$
\end{Proposition}

\begin{proof} Define filtrations  of ideals $\{J_i\}$ and  $\{J(a)_i\}$ by $J_i=I(1)_{in_1}\cdots I(r)_{in_r}$ and $J(a)_i=I_a(1)_{in_1}\cdots I_a(r)_{in_r}$.

We use a construction and method from the proof of  \cite[Theorem 4.2]{C2}. We begin by reviewing the construction in the context of this proposition. Since $\lambda(R/J_i)=\lambda_{\hat R}(\hat R/\hat J_i)$ and $\lambda(R/J(a)_i)=\lambda_{\hat R}(\hat R/J(a)_i\hat R)$ for all $i$ and $a$  and $\hat R$ is a domain, we may assume that $R$ is complete and thus is excellent. 
By Lemma \ref{Res}, there exists a regular local ring $S$ of dimension $d$ which birationally dominates $R$.
Choosing a regular system of parameters $y_1,\ldots,y_d$ in $S$ and $\lambda_1,\ldots,\lambda_d\in \RR$ which are rationally independent real numbers such that $\lambda_i\ge 1$ for all $i$, we define a valuation $\nu$ on the quotient field of $R$ such that $\nu$  dominates $S$ by prescribing $\nu(y_1^{a_1}\cdots y_d^{a_d})=a_1\lambda_1+\cdots+a_d\lambda_d$ for $a_1,\ldots,a_d\in \NN$ and $\nu(\gamma)=0$ if $\gamma\in S$ is a unit. Let $k=R/m_R$ and $k'=S/m_S$. 

We will show that the residue field $V_{\nu}/m_{\nu}=k'$. Given an element $h\in V_{\nu}$, let $[h]$ denote the class of $f$ in the residue field $V_{\nu}/m_{\nu}$. Suppose $h\in V_{\nu}$ and $\nu(h)=0$. Write $h=\frac{f}{g}$ with $f,g\in S$. There exist a unit $\alpha\in S$, $i_1,\ldots,i_d\in \NN$ and $a\in S$ such that $f=\alpha y_1^{i_1}\cdots y_d^{i_d}+a$ and
$\nu(a)>\nu(f)=i_1\lambda_1+\cdots+i_d\lambda_d$. Similarly, there exist a unit $\beta\in S$, $j_1,\ldots,j_d\in \NN$ and $b\in S$ such that $g=\beta y_1^{j_1}\cdots y_d^{j_d}+b$ with $\nu(b)>\nu(g)=j_1\lambda_1+\cdots+j_d\lambda_d$.  We have that $y_1^{i_1}\cdots y_d^{i_d}=y_1^{j_1}\cdots y_d^{j_d}$ since
 $\nu(f)=\nu(g)$. Now $[\frac{f}{y_1^{i_1}\cdots y_d^{i_d}}]=[\alpha]$ and $[\frac{g}{y_1^{i_1}\cdots y_d^{i_d}}]=[\beta]$ so $[h]=\frac{[\alpha]}{[\beta]}\in S/m_S=k'$.

For $\lambda\in \RR_{\ge 0}$, define ideals $K_{\lambda}$ and $K_{\lambda}^+$ in the valuation ring $V_{\nu}$ of $\nu$ by
$$
K_{\lambda}=\{f\in {\rm Q}(R)\mid \nu(f)\ge \lambda\}
$$
and
$$
K_{\lambda}^+=\{f\in {\rm Q}(R)\mid \nu(f)>\lambda\}.
$$
For $t\ge 1$, define semigroups 
$$
\begin{array}{lll}
\Gamma^{(t)}&=&\{(m_1,\ldots,m_d,i)\in \NN^{d+1}\mid \dim_k J_i\cap K_{m_1\lambda_1+\cdots+m_d\lambda_d}/J_i\cap K^+_{m_1\lambda_1+\cdots+m_d\lambda_d} \ge t\\
&&\mbox{ and }m_1+\cdots+m_d\le \beta i\},
\end{array}
$$
$$
\begin{array}{lll}
\Gamma(a)^{(t)}&=&
\{(m_1,\ldots,m_d,i)\in \NN^{d+1}\mid \dim_k J(a)_i\cap K_{m_1\lambda_1+\cdots+m_d\lambda_d}/J(a)_i\cap K^+_{m_1\lambda_1+\cdots+m_d\lambda_d} \ge t\\
&&\mbox{ and }m_1+\cdots+m_d\le \beta i\}
\end{array}
$$
and
$$
\begin{array}{lll}
\hat\Gamma^{(t)}&=&\{(m_1,\ldots,m_d,i)\in \NN^{d+1}\mid \dim_k R\cap K_{m_1\lambda_1+\cdots+m_d\lambda_d}/R\cap K^+_{m_1\lambda_1+\cdots+m_d\lambda_d} \ge t\\
&&\mbox{ and }m_1+\cdots+m_d\le \beta i\}.
\end{array}
$$
Here $\beta=\alpha c$ where $c\in \ZZ_+$ is chosen so that $m_R^c\subset J_1=J(a)_1=I(1)_{n_1}\cdots I(r)_{n_r}$
 and $\alpha\in \ZZ_+$ is such that $K_{\alpha n}\cap R\subset m_R^n$ for all $n\in \NN$. Such an $\alpha$ exists by  \cite[Lemma 4.3]{C1}.
 Define $\Gamma_m^{(t)}=\Gamma^{(t)}\cap (\NN^d\times\{m\})$,  $\Gamma(a)^{(t)}_m=\Gamma(a)^{(t)}\cap (\NN^d\times\{m\})$ and $\hat\Gamma^{(t)}_m=\hat\Gamma^{(t)}\cap (\NN^d\times\{m\})$ for $m\in \NN$.

The Newton-Okounkov body of a (strongly nonnegative) sub semigroup $S$ of $\ZZ^d\times \NN$ is defined as
$$
\Delta(S)={\rm con}(S)\cap (\RR^d\times \{1\})
$$
 where ${\rm con}(S)$ is the closed convex cone which is the closure of the set of all linear combinations $\sum\lambda_i s_i$  with $s_i\in S$ and $\lambda_i$ a nonnegative real number. This theory is developed in \cite{Ok}, \cite{LM} and \cite{KK} and is summarized in  \cite[Section 3]{C2}.

By \cite[Lemmas 4.5 and  4.6]{C2} and \cite[Theorem 3.2]{C2}, 
\begin{equation}\label{eq13}
\lim_{m\rightarrow\infty}\frac{\#\Gamma_m^{(t)}}{m^d}={\rm Vol}(\Delta(\Gamma^{(t)})),
\end{equation}
\begin{equation}\label{eq9}
\lim_{m\rightarrow\infty}\frac{\#\Gamma(a)_m^{(t)}}{m^d}={\rm Vol}(\Delta(\Gamma(a)^{(t)}))
\end{equation}
and
\begin{equation}\label{eq10}
\lim_{m\rightarrow\infty}\frac{\#\hat\Gamma_m^{(t)}}{m^d}={\rm Vol}(\Delta(\hat\Gamma^{(t)}))
\end{equation}
all exist (where $\#T$ is the number of elements in a finite set $T$).

By \cite[(19) on page 11]{C2}, 
\begin{equation}\label{eq7}
\begin{array}{lll}
F_a(n_1,\ldots,n_r)&=&
\lim_{m\rightarrow \infty}\frac{\lambda(R/J(a)_m)}{m^d}\\
&=& \sum_{t=1}^{[k':k]}\lim_{m\rightarrow \infty}\frac{\#\hat\Gamma^{(t)}_m}{m^d}-\sum_{t=1}^{[k':k]}\lim_{m\rightarrow\infty}\frac{\#\Gamma(a)^{(t)}_m}{m^d}
\end{array}
\end{equation}
with a similar formula
\begin{equation}\label{eq8}
\begin{array}{lll}
F(n_1,\ldots,n_r)&=&
\lim_{m\rightarrow \infty}\frac{\lambda(R/J_m)}{m^d}\\
&=& \sum_{t=1}^{[k':k]}\lim_{m\rightarrow \infty}\frac{\#\hat\Gamma^{(t)}_m}{m^d}-\sum_{t=1}^{[k':k]}\lim_{m\rightarrow\infty}\frac{\#\Gamma^{(t)}_m}{m^d}.
\end{array}
\end{equation}
Let
$$
\overline a=\lfloor a/\max\{n_1,\ldots,n_r\}\rfloor
$$
where $\lfloor x\rfloor$ is the greatest integer in a real number $x$.
We have that  
\begin{equation}\label{eq11}
\Gamma_i^{(t)}=\Gamma(a)_i^{(t)}\mbox{ for }i\le \overline a
\end{equation}
and so
$$
n*\Gamma_{\overline a}^{(t)}:=\{x_1+\cdots+x_n\mid x_1,\ldots,x_n\in \Gamma_{\overline a}^{(t)}\} \subset \Gamma(a)^{(t)}_{n\overline a}\mbox{ for all }n\ge 1.
$$
By  \cite[Proposition 3.1]{LM} (recalled in \cite[Theorem 3.3]{C2}) and since $\overline a\mapsto\infty$ as $a\mapsto\infty$, given $\epsilon>0$, there exists $a_0>0$ such that for all $a\ge a_0$ we have
\begin{equation}\label{eq12}
\begin{array}{lll}
{\rm Vol}(\Delta(\Gamma^{(t)}))&\ge &{\rm Vol}(\Delta(\Gamma(a)^{(t)})
=\lim_{n\rightarrow\infty}\frac{\#\Gamma(a)^{(t)}_n}{n^d} =\lim_{n\rightarrow\infty}\frac{\#\Gamma(a)^{(t)}_{n\overline a}}{(n\overline a)^d}\\
&\ge &\lim_{n\rightarrow\infty}\frac{\#(n*\Gamma_{\overline a}^{(t)})}{(n\overline a)^d}\ge {\rm Vol}(\Delta(\Gamma^{(t)}))-\epsilon.
\end{array}
\end{equation}
By (\ref{eq7}) - (\ref{eq12}), the proposition holds. 

\end{proof}

The following corollary now follows from Lemma \ref{Prop1} and Proposition \ref{Prop2}.
\begin{Corollary}\label{Cor1} For all $i_1,\ldots,i_r\in\NN$ with $i_1+\cdots+i_r=d$,
\begin{equation}\label{eq5}
b_{i_1,\ldots,i_r}:= 
\lim_{a\rightarrow\infty} b_{i_1,\ldots,i_r}(a)
\end{equation}
exists (in $\RR$).
\end{Corollary}

Now define a homogeneous polynomial 
$$
G(x_1,\ldots,x_r)=\sum_{i_1+\cdots+i_r=d}b_{i_1,\ldots,i_r}x_1^{i_1}\cdots x_r^{i_r}\in \RR[x_1,\ldots,x_r],
$$
where the $b_{i_1,\ldots,i_r}$ are defined by (\ref{eq5}).

\begin{Theorem}\label{Theorem1} For all $n_1,\ldots,n_r\in \ZZ_+$,
$$
F(n_1,\ldots,n_r)=G(n_1,\ldots,n_r).
$$
\end{Theorem}

\begin{proof} For fixed $n_1,\ldots,n_r\in \ZZ_+^r$ and $a\in \ZZ_+$,
$$
\begin{array}{lll}
|F(n_1,\ldots,n_r)-G(n_1,\ldots,n_r)|&\le&
|F(n_1,\ldots,n_r)-F_a(n_1,\ldots,n_r)|\\
&&+\,|F_a(n_1,\ldots,n_r)-G(n_1,\ldots,n_r)|
\end{array}
$$
which is arbitrarily small for $a\gg 0$ by Proposition \ref{Prop2} and Corollary \ref{Cor1}.
\end{proof}

\section{Reduction to  local domains}\label{Reduc}

\begin{Lemma}\label{Lemma2}  Suppose $R$ is a Noetherian domain and $M$ is a torsion free finitely generated $R$-module. Then there exists a short exact sequence of $R$-modules
$$
0\rightarrow R^s\rightarrow M\rightarrow F\rightarrow 0
$$
where $s={\rm rank}(M)$ and $\dim F<\dim R$.
\end{Lemma}

\begin{proof} Let $K$ be the quotient field of $R$ and $\{e_1,\ldots,e_s\}$ be a $K$-basis of $M\otimes_RK$. Since $M$ is torsion free, we  have a natural inclusion $M\subset M\otimes K$. For all $i$, there exists $0\ne x_i\in R$ such that $x_ie_i\in M$, so after replacing $e_i$ with $x_ie_i$, we may assume that $e_i\in M$. Let $\phi:R^s\rightarrow M$ be the $R$-module homomorphism $\phi=(e_1,\ldots,e_s)$.  Let $L$ be the kernel of $\phi$ and $F$ be the cokernel. We have a commutative diagram
$$
\begin{array}{ccccccccccc}
0&\rightarrow&L&\rightarrow&R^s&\stackrel{\phi}{\rightarrow}&M&\rightarrow &F&\rightarrow &0\\
&&&&\downarrow&&\downarrow&&&&\\
&&&&K^s&\stackrel{\phi}{\rightarrow}&M\otimes_RK&\rightarrow&F\otimes_RK&\rightarrow&0
\end{array}
$$
where the vertical arrows are injective and the rows are  exact. By our construction of $\phi$, $K^s\stackrel{\phi}{\rightarrow}M\otimes_RK$ is an isomorphism. Thus $L=0$ and $\dim F<\dim R$.

\end{proof}

\begin{Lemma}\label{Lemma3} Suppose $R$ is a Noetherian local ring  of dimension $d$ and  $M$ is a finitely generated $R$-module. Let $T$ be a submodule of  $M$ such that $\dim T<d$, so that there is a short exact sequence of $R$-modules
$$
0\rightarrow T\rightarrow M\rightarrow M/T:=\overline M\rightarrow 0.
$$
Suppose $\mathcal I(1)=\{I(1)_i\},\ldots,\mathcal I(r)=\{I(r)_i\}$ are filtrations of $R$ by $m_R$-primary ideals. 
Then for fixed $n_1,\ldots,n_r\in \NN$,
$$
\lim_{m\rightarrow \infty} \frac{\lambda(M/I(1)_{mn_1}\cdots I(r)_{mn_r}M)}{m^d} =\lim_{m\rightarrow\infty}\frac{\lambda(\overline M/I(1)_{mn_1}\cdots I(r)_{mn_r}\overline M)}{m^d}.
$$
\end{Lemma}

\begin{proof} Define a filtration of $R$ by $m_R$-primary ideals by $J_m=I(1)_{mn_1}\cdots I(r)_{mn_r}$. We have short exact sequences of $R$-modules
$$
0\rightarrow T/T\cap (J_mM)\rightarrow M/J_mM\rightarrow \overline M/J_m\overline M\rightarrow 0.
$$
There exists a positive integer $c$ such that $m_R^{c}\subset J_1$. Thus $m_R^{cm}T\subset T\cap (J_mM)$ for all $m$ and
$$
\lambda(T/T\cap J_mM)\le \lambda(T/m_R^{cm}T).
$$
Since $\dim T< d$,
$$
\lim_{m\rightarrow \infty}\frac{\lambda(T/m_R^{cm}T)}{m^d}=0
$$
and the lemma follows.

\end{proof}

\begin{Lemma}\label{Lemma4} Suppose that $R$ is a Noetherian local domain of dimension $d$ and $M$ is a  finitely generated $R$-module. Suppose $\mathcal I(1)=\{I(1)_i\},\ldots,\mathcal I(r)=\{I(r)_i\}$ are filtrations of $R$ by $m_R$-primary ideals. Let $s={\rm rank}(M)$. Suppose $n_1,\ldots,n_r\in \NN$. Then
$$
\lim_{m\rightarrow \infty}\frac{\lambda(M/I(1)_{mn_1}\cdots I(r)_{mn_r}M)}{m^d}
=s\left(\lim_{m\rightarrow \infty}\frac{\lambda(R/I(1)_{mn_1}\cdots I(r)_{mn_r})}{m^d}\right).
$$
\end{Lemma}

\begin{proof} Define a filtration of $m_R$-primary ideals by  $J_m=I(1)_{mn_1}\cdots I(r)_{mn_r}$. By Lemma \ref{Lemma3}, we may assume that $M$ is torsion free, so there exists by Lemma \ref{Lemma2}, a short exact sequence of $R$-modules
$$
0\rightarrow R^s\rightarrow M\rightarrow F\rightarrow 0
$$
where $\dim F<d$. There exists $c>0$ such that $m_R^c\subset J_1$. There exists $0\ne x\in R$ such that $xM\subset R^s$.  We have exact sequences for all $m\in \ZZ_+$,
$$
\begin{array}{l}
0\rightarrow R^s\cap(J_mM)/J_mR^s\rightarrow
R^s/J_mR^s\\
\rightarrow M/J_mM\rightarrow N_m\rightarrow 0
\end{array}
$$
where $N_m$ is defined to be the cokernel of the last map, and we have an exact sequence
\begin{equation}\label{eq20}
0\rightarrow A_m\rightarrow R^s/J_mR^s\stackrel{x}{\rightarrow} R^s/J_mR^s\rightarrow W_m\rightarrow 0
\end{equation}
where $A_m$ is the kernel of the first map and $W_m$ is the cokernel of the last map. We have 
$$
A_m=\left[(J_m:x)/J_m\right]^s.
$$
We have that 
$$
x(R^s\cap J_mM)\subset J_mR^s
$$
so that 
$$
\lambda(R^s\cap (J_mM)/J_mR^s)\le \lambda(A_m).
$$
We have that 
$$
W_m\cong [(R/(x))/J_m(R/(x))]^s
$$
so
$$
\lambda(W_m)\le\lambda((R/(x))/m_R^{cm}(R/(x)))^s
$$
for all $m$. Thus 
$$
\lim_{m\rightarrow\infty}\frac{\lambda(A_m)}{m^d}=\lim_{m\rightarrow\infty}\frac{\lambda(W_m)}{m^d}=0
$$
by (\ref{eq20}) and so
$$
\lim_{m\rightarrow\infty}\frac{\lambda(R^s\cap J_mM/J_mR^s)}{m^d}=0.
$$
Now $xM\subset R^s$ implies
$$
N_m\cong M/R^s+J_mM=M/(R^s+J_mM+xM).
$$
Thus
$$
\lambda(N_m)\le \lambda((M/xM)/m_R^{cm}(M/xM)
$$
and so 
$$
\lim_{m\rightarrow \infty}\frac{\lambda(N_m)}{m^d}=0
$$
since $\dim M/xM<d$, and the lemma follows.

\end{proof}

\begin{Lemma}\label{Lemma5} Suppose that $R$ is a $d$-dimensional reduced Noetherian local ring and $M$ is a finitely generated $R$-module. Let $\{P_1,\ldots,P_s\}$ be the minimal primes of $R$ and $S=\bigoplus_{i=1}^sR/P_i$. Suppose $\mathcal I(1)=\{I(1)_i\},\ldots,\mathcal I(r)=\{I(r)_i\}$ are filtrations of $R$ by $m_R$-primary ideals. Suppose that $n_1,\ldots,n_r\in \ZZ_+$ are fixed. Then
$$
\lim_{m\rightarrow \infty}\frac{\lambda(M/I(1)_{mn_1}\cdots I(r)_{mn_r})}{m^d}
=\lim_{m\rightarrow\infty}\frac{\lambda(M\otimes_RS/I(1)_{mn_1}\cdots I(r)_{mn_r}M\otimes_RS)}{m^d}.
$$
\end{Lemma}

\begin{proof} Define a filtration of $R$ by $m_R$-primary ideals by $J_m=I(1)_{mn_1}\cdots I(r)_{mn_r}$.  There exists $c\in \ZZ_+$ such that $m_R^c\subset J_1$. Since $S$ is a finitely generated $R$ submodule of the total ring of fractions $T=\bigoplus_{i=1}^s Q(R/P_i)$ of $R$, there exists a non zerodivisor $x\in R$ such that $xS\subset R$. Tensoring the short exact sequence 
$$
0\rightarrow R\rightarrow S\rightarrow S/R\rightarrow 0
$$
of $R$-modules with $M$, 
we  have a natural short exact sequence of $R$-modules,
$$
M\stackrel{\gamma}{\rightarrow} M\otimes_RS\rightarrow M\otimes_R(S/R)\rightarrow 0.
$$
Let $K=\mbox{kernel }\gamma$ and $U=\mbox{Image }\gamma$. We have that $(\mbox{Kernel }\gamma)_{P_i}=0$ for $1\le i\le s$ since $R_{P_i}\cong S_{P_i}$ for all $i$. Thus $\dim \mbox{Kernel }\gamma<d$, and by Lemma \ref{Lemma3},
$$
\lim_{m\rightarrow\infty} \frac{\lambda(U/J_mU)}{m^d}=\lim_{m\rightarrow\infty}\frac{\lambda(M/J_mM)}{m^d}.
$$
Let $V=M\otimes_RS$. We have short exact sequences of $R$-modules,
$$
0\rightarrow U\cap J_mV/J_mU\rightarrow U/J_mU\rightarrow V/J_mV\rightarrow N_m\rightarrow 0
$$
where  $N_m=V/U+J_mV$. We also have short exact sequences
\begin{equation}\label{eq22}
0\rightarrow A_m\rightarrow U/J_mU\stackrel{x}{\rightarrow} U/J_mU\rightarrow W_m\rightarrow 0
\end{equation}
where $A_m$ is the kernel of multiplication by $x$ and $W_m$ is the cokernel. Now $x(U\cap J_mV)\subset J_mU$, so $U\cap J_mV/J_mU\subset A_m$ for all $m$. Now $W_m\cong (U/xU)/J_m(U/xU)$ and $\dim U/xU<d$. We have that
$$
\lambda(W_m)\le \lambda((U/xU)/m_R^{mc}(U/xU))
$$
and thus 
$$
\lim_{m\rightarrow\infty}\frac{\lambda(W_m)}{m^d}=0.
$$
From (\ref{eq22}), we have 
$$
\lim_{m\rightarrow\infty}\frac{\lambda(U\cap J_mV/J_mU)}{m^d}\le
\lim_{m\rightarrow \infty}\frac{\lambda(A_m)}{m^d}=\lim_{m\rightarrow\infty}\frac{\lambda(W_m)}{m^d}=0.
$$
Since $xV\subset U$, we have
$$
N_m\cong V/U+J_mV=V/(U+J_mV+xV).
$$
Thus
$$
\lambda(N_m)\le\lambda((V/xV)/m_R^{mc}(V/xV))
$$
for all $m$, so 
$$
\lim_{m\rightarrow\infty}\frac{\lambda(N_m)}{m^d}=0
$$
since $\dim V/xV<d$.
\end{proof}

\section{Mixed Multiplicities of Filtrations}\label{MMult}

The following theorem allows  us to define  mixed multiplicities for arbitrary (not necessarily Noetherian) filtrations of $m_R$-ideals in a Noetherian local ring with $\dim N(\hat R)<\dim R$.  By Theorem \ref{TheoremI20}, if the assumption $\dim N(\hat R)<d$ is removed from the hypotheses of Theorem \ref{Theorem2}, then the conclusions of Theorem \ref{Theorem2} will no longer be true. 
Theorem \ref{Theorem2} generalizes a theorem of  Bhattacharya \cite{Bh} and Teissier and Risler \cite{T1} (also proven in  \cite[Theorem 17.4.2]{HS}) for $m_R$-primary ideals to filtrations of $m_R$-primary ideals.

\begin{Theorem}\label{Theorem2} Suppose that $R$ is a Noetherian local ring of dimension $d$ such that 
$$
\dim N(\hat R)<d
$$
 and $\mathcal I(1)=\{I(1)_i\},\ldots,\mathcal I(r)=\{I(r)_i\}$ are (not necessarily Noetherian) filtrations of $R$ by $m_R$-primary ideals. Suppose that $M$ is a finitely generated $R$-module. Then there exists a homogeneous polynomial $G(x_1,\ldots,x_r)\in \RR[x_1,\ldots,x_r]$ which is of total degree $d$  if $G$ is nonzero, such that for all $n_1,\ldots,n_r\in \ZZ_+$,
$$
\lim_{m\rightarrow\infty}\frac{\lambda(M/I(1)_{mn_1}\cdots I(r)_{nm_r}M)}{m^d}=G(n_1,\ldots,n_r). 
$$
 \end{Theorem}

We will see in Theorem  \ref{Cor2} that the conclusions of the theorem hold for all $n_1,\ldots,n_r\in \NN$.

\begin{proof} Replacing $R$ with $\hat R$, $I(j)_i$ with $I(j)_i \hat R$  and $M$ with $M\otimes_R\hat R$, we may assume that $R$ is complete. By Lemma \ref{Lemma3} (taking $T=N(R)M$) we reduce to the case where $R$ is analytically unramified. By Lemma \ref{Lemma5}, we reduce to the case where $R$ is analytically irreducible. By Lemma \ref{Lemma4}, we reduce to the case where $R$ is analytically irreducible and $M=R$. Theorem \ref{Theorem2} now follows from Theorem \ref{Theorem1}.
\end{proof}

Let assumptions be as the statement of Theorem \ref{Theorem2}. 
Generalizing the classical definition of mixed multiplicities for  $m_R$-primary ideals (\cite{Bh}, \cite{R}, \cite{T1},  \cite[Definition 17.4.3]{HS}) 
we define the mixed multiplicities  of $M$ of type  $(d_1,\ldots,d_r)$ with respect to  the filtrations $\mathcal I(1),\ldots,\mathcal I(r)$ of $R$ by $m_R$-primary ideals
$$
e_R(\mathcal I(1)^{[d_1]},\ldots, \mathcal I(r)^{[d_r]};M)
$$
  from the coefficients of the homogeneous polynomial $G(n_1,\ldots,n_r)$. Specifically,   
 we write 
$$
G(n_1,\ldots,n_r)=\sum_{d_1+\cdots d_r=d}\frac{1}{d_1!\cdots d_r!}e_R(\mathcal I(1)^{[d_1]},\ldots, \mathcal I(r)^{[d_r]};M)n_1^{d_1}\cdots n_r^{d_r}.
$$

We write the multiplicity $e_R(\mathcal I;M)=e_R(\mathcal I^{[d]};M)$ if $r=1$, and $\mathcal I=\{I_i\}$ is a filtration of  $R$ by $m_R$-primary ideals. We have that
$$
e_R(\mathcal I;M)=\lim_{m\rightarrow \infty}d!\frac{\lambda(M/I_mM)}{m^d}.
$$

\begin{Proposition}\label{Prop6} Suppose that $R$ is a $d$-dimensional Noetherian local ring with  $\dim N(\hat R)<d$. Suppose $\mathcal I(j)=\{I(j)_i\}$ for $1\le j\le r$ are filtrations of $R$ by $m_R$-primary ideals and $M$ is a finitely generated $R$-module. Then for all $d_1,\ldots,d_r$ with $d_1+\cdots +d_r=d$, we have that
$$
\lim_{a\rightarrow\infty}e_R(\mathcal I_a(1)^{[d_1]},\ldots,\mathcal I_a(r)^{[d_r]};M)=e_R(\mathcal I(1)^{[d_1]},\ldots,\mathcal I(r)^{[d_r]};M).
$$
\end{Proposition}

\begin{proof} The proof of Theorem \ref{Theorem2} gives a reduction  to the case that $R$ is analytically irreducible  and $M=R$. The proposition now follows from Corollary \ref{Cor1}.
\end{proof}

The following theorem extends to filtrations of $R$ by $m_R$-primary ideals the Minkowski inequalities of $m_R$-primary ideals of Teissier \cite{T1}, \cite{T2} and Rees and Sharp \cite{RS}.
The inequality 4) of Theorem \ref{Theorem12} was proven for graded families of $m_R$-primary ideals in a regular local ring with algebraically closed residue field by Musta\c{t}\u{a} (Corollary 1.9 \cite{Mus}) and more recently by Kaveh and Khovanskii (\cite[Corollary 7.14]{KK1}). The inequality 4) was proven with our assumption that $\dim N(\hat R)<d$ in \cite[Theorem 3.1]{C3}.
Inequalities 2) - 4) can be deduced directly from inequality 1), as in the proof of  \cite[Corollary 17.7.3]{HS}, as explained in  \cite{T2}, \cite{RS} and \cite{HS}.

\begin{Theorem}\label{Theorem12}(Minkowski Inequalities)  Suppose that $R$ is a Noetherian $d$-dimensional  local ring with $\dim N(\hat R)<d$, $M$ is a finitely generated $R$-module and $\mathcal I(1)=\{I(1)_j\}$ and $\mathcal I(2)=\{I(2)_j\}$ are filtrations of $R$ by $m_R$-primary ideals. Then 
\begin{enumerate}
\item[1)] $e_R(\mathcal I(1)^{[i]},\mathcal I(2)^{[d-i]};M)^2\le e_R(\mathcal I(1)^{[i+1]},\mathcal I(2)^{[d-i-1]};M)e_R(\mathcal I(1)^{[i-1]},\mathcal I(2)^{[d-i+1]};M)$

 for $1\le i\le d-1$.
\item[2)]  For $0\le i\le d$,

 $e_R(\mathcal I(1)^{[i]},\mathcal I(2)^{[d-i]};M)e_R(\mathcal I(1)^{[d-i]},\mathcal I(2)^{[i]};M)\le e_R(\mathcal I(1);M)e_R(\mathcal I(2);M)$,
\item[3)] For $0\le i\le d$, $e_R(\mathcal I(1)^{[d-i]},\mathcal I(2)^{[i]};M)^d\le e_R(\mathcal I(1);M)^{d-i}e_R(\mathcal I(2);M)^i$ and
\item[4)]  $e_R(\mathcal I(1)\mathcal I(2));M)^{\frac{1}{d}}\le e_R(\mathcal I(1);M)^{\frac{1}{d}}+e_R(\mathcal I(2);M)^{\frac{1}{d}}$, 

where $\mathcal I(1)\mathcal I(2)=\{I(1)_jI(2)_j\}$.
\end{enumerate}
\end{Theorem}

\begin{proof} By the reduction of the proof of Theorem \ref{Theorem2}, it suffices to prove the theorem for $R$ an analytically irreducible domain and $M=R$. We first will show that for all $a\in \ZZ_+$, the Minkowski inequalities hold for the $a$-th truncated filtrations  $\mathcal I_a(1)=\{I_a(1)_m\}$ and $\mathcal I_a(2)=\{I_a(2)_m\}$ (defined in Definition \ref{trunc}). 

Given $a\in \ZZ_+$, there exists $f_a\in \ZZ_+$ such that $I_a(i)_{f_am}=(I_a(i)_{f_a})^m$ for all $m\ge 0$ and $i=1,2$.   Define  filtrations of $R$ by $m_R$-primary ideals by $J_a(i)_m=I_a(i)_{f_am}$. Then for $n_1, n_2\in \ZZ_+$, 
$$
\lim_{m\rightarrow\infty}\frac{\lambda(R/J_a(1)_{mn_1}J_a(2)_{mn_2})}{m^d}=
\sum_{d_1+d_2=d}\frac{1}{d_1!d_2!}e_R(J_a(1)_1^{[d_1]},J_a(2)_1^{[d_2]};R)n_1^{d_1} n_2^{d_2},
$$
$$
\lim_{m\rightarrow \infty}\frac{\lambda(R/J_a(k)_1^m)}{m^d}=\frac{1}{d!}e_R(J_a(k)_1;R)
$$
for $k=1$ and 2 and
$$
\lim_{m\rightarrow \infty}\frac{\lambda(R/(J_a(1)_1J_a(2)_1)^m)}{m^d}=\frac{1}{d!}e_R(J_a(1)_1J_a(2)_1;R)
$$
where $e_R(J_a(1)^{[d_1]}$, $J_a(2)^{[d_2]};R)$, e$_R(J_a(1)_1;R)$, $e_R(J_a(2)_1;R)$, $e_R(J_a(1)J_a(2)_1;R)$  are the usual mixed multiplicities of ideals (\cite[Theorem 17.4.2, Definition 17.4.3]{HS}).

Now the Minkowski inequalities hold for the  mixed multiplicities of ideals 
$$
e_R(J_a(1)_1^{[d_1]},J_a(2)_1^{[d_2]};R), e_R(J_a(1)_1;R), e_R(J_a(2)_1;R)\mbox{ and }e_R(J_a(1)_1J_a(2)_1;R)
$$
  by \cite{RS} or  \cite[Theorem 17.7.2 and Corollary 17.7.3]{HS}. By Lemma \ref{Lemma0},
$$
\lim_{m\rightarrow \infty}\frac{\lambda(R/I_a(1)_{mn_1},I_a(2)_{mn_2})}{m^d}
=\frac{1}{f_a^d}\left(\lim_{m\rightarrow \infty}\frac{\lambda(R/J_a(1)_1^{mn_1}J_a(2)_1^{mn_2})}{m^d}\right)
$$
for all $n_1,n_2\in \NN$,
$$
\lim_{m\rightarrow \infty}\frac{\lambda(R/I_a(k)_m)}{m^d}=\frac{1}{f_a^d}\lim_{m\rightarrow\infty}\frac{\lambda(R/J_a(k)_1^m)}{m^d}
$$
for $k=1$ and 2 and
$$
\lim_{m\rightarrow \infty}\frac{\lambda(R/I_a(1)_mI_a(2)_m)}{m^d}=\frac{1}{f_a^d}\lim_{m\rightarrow \infty}\frac{\lambda(R/(J_a(1)_1J_a(2)_1)^m)}{m^d}.
$$
By Lemma \ref{Prop1},
$$
e_R(\mathcal I_a(1)^{[d_1]},\mathcal I_a(2)^{[d_2]};R)=\frac{1}{f_a^d}e_R(J_a(1)_1^{[d_1]},J_a(2)_1^{[d_2]};R)
$$
for all $d_1,d_2$,
$$
e_R(\mathcal I_a(1);R)=\frac{1}{f_a^d}e_R(J_a(1)_1;R), e_R(\mathcal I_a(2)_1;R)=\frac{1}{f_a^d}e_R(J_a(2)_1;R)
$$
and
$$
e_R(\mathcal I_a(1)\mathcal I_a(2);R)=\frac{1}{f_a^d}e_R(J_a(1)_1J_a(2)_1;R).
$$

Thus the Minkowski inequalities hold for the $e_R(\mathcal I_a(1)^{[d_1]},\mathcal I_a(2)^{[d_2]};R)$, $e_R(\mathcal I_a(1);R)$, $e_R(\mathcal I_a(2);R)$ and $e_R(\mathcal I_a(1)\mathcal I_a(2);R)$. Now the Minkowski inequalities hold for the 
$$
e_R(\mathcal I(1)^{[d_1]},\mathcal I(2)^{[d_2]};R), e_R(\mathcal I(1);R), e_R(\mathcal I(2);R)\mbox{ and }e_R(\mathcal I(1)\mathcal I(2);R)
$$
  by Proposition \ref{Prop6}.
\end{proof}

\begin{Remark}(Minkowski equality)
Teissier \cite{T3} (for Cohen Macaulay normal complex analytic $R$), Rees and Sharp \cite{RS} (in dimension 2) and Katz \cite{Ka} (in complete generality) have proven that if $R$ is a $d$-dimensional formally equidimensional Noetherian local ring and $I(1)$, $I(2)$ are $m_R$-primary ideals such that the Minkowski equality
$$
e_R((I(1) I(2));R)^{\frac{1}{d}}= e_R( I(1);R)^{\frac{1}{d}}+e_R(I(2);R)^{\frac{1}{d}}
$$
holds,
 then there exist positive integers $r$ and $s$ such that the complete ideals $\overline {I(1)^r}$ and $\overline{I(2)^s}$ are equal, which is equivalent to the statement that the $R$-algebras $\bigoplus_{n\ge 0}I(1)^n$ and $\bigoplus_{n\ge 0}I(2)^n$ have the same integral closure. 
 
 This statement is not true for filtrations, even in a regular local ring, as is shown by the following simple example. Let $k$ be a field and $R$ be the power series ring $R=k[[x_1,\ldots,x_d]]$. Let $\mathcal I(1)=\{I(1)_i\}$ where $I(1)_i=m_R^i$ and 
 $\mathcal I(2)=\{I(2)_i\}$ where $I(2)_i=m_R^{i+1}$. Then the Minkowski equality
 $$
e_R((\mathcal I(1) \mathcal I(2));R)^{\frac{1}{d}}= e_R( \mathcal I(1);R)^{\frac{1}{d}}+e_R(\mathcal I(2);R)^{\frac{1}{d}}
$$ 
is satisfied  but  $\bigoplus_{i\ge 0}I(1)_i$ and $\bigoplus_{i\ge 0}I(2)_i$ do not have the same integral closure.
\end{Remark}

The following  proposition generalizes  an identity of Rees, \cite[Lemma 2.4]{R}.

\begin{Proposition}\label{Meq}  Suppose that $R$ is a Noetherian local ring of dimension $d$ such that $\dim N(\hat R)<d$ and $\mathcal I(1)=\{I(1)_i\},\ldots,\mathcal I(r)=\{I(r)_i\}$ are  filtrations of $R$ by $m_R$-primary ideals. Suppose that $M$ is a finitely generated $R$-module. Then for $1\le i\le r$,
$$
\begin{array}{l}
e_R(\mathcal I(1)^{[d_1]},\cdots, \mathcal I(i-1)^{[d_{i-1}]},\mathcal I(i)^{[0]},\mathcal I(i+1)^{[d_{i+1}]},\cdots,\mathcal I(r)^{[d_r]};M)\\
=e_R(\mathcal I(1)^{[d_1]},\cdots, \mathcal I(i-1)^{[d_{i-1}]},\mathcal I(i+1)^{[d_{i+1}]},\cdots,\mathcal I(r)^{[d_r]};M)
\end{array}
$$
whenever $d_1+\cdots+d_{i-1}+d_{i+1}\cdots+d_r=d$.

In particular, 
$$
e_R(\mathcal I(i);M)=e_R(\mathcal I(1)^{[0]},\ldots,\mathcal I(i-1)^{[0]},\mathcal I(i)^{[d]},\mathcal I(i+1)^{[0]},\ldots,\mathcal I(r)^{[0]};M).
$$
\end{Proposition}

\begin{proof} By the proof of Theorem \ref{Theorem12}, we need only show that the identities hold for $m_R$-primary ideals $I(1),\ldots,I(r)$. We may assume that $i=r$. Let $G(x_1,\ldots,x_r)\in \QQ[x_1,\ldots,x_r]$ be the homogeneous polynomial of degree $d$ such that 
$$
\lim_{m\rightarrow \infty}
\frac{\lambda(M/I(1)^{mn_1}\cdots I(r)^{mn_r}M)}{m^d}=G(n_1,\ldots,n_r)
$$
whenever $n_1,\ldots,n_r\in \ZZ_+$, and let $Q(x_1,\ldots,x_{r-1})\in \QQ[x_1,\ldots,x_{r-1}]$ be the homogeneous polynomial of degree $d$ such that 
$$
\lim_{m\rightarrow \infty}
\frac{\lambda(M/I(1)^{mn_1}\cdots I(r-1)^{mn_{r-1}}M)}{m^d}=Q(n_1,\ldots,n_{r-1})
$$
whenever $n_1,\ldots,n_{r-1}\in \ZZ_+$.  Then for all $n_1,\ldots,n_{r-1}\in \ZZ_+$,
$$
\lim_{m\rightarrow \infty}\frac{G(mn_1,\ldots,mn_{r-1},1)}{m^d}=\lim_{m\rightarrow \infty}\frac{G(mn_1,\ldots,m n_{r-1},0)}{m^d}
$$
and for $\alpha\in \ZZ_+$,
$$
\lim_{m\rightarrow \infty}\frac{Q(mn_1,\ldots,mn_{r-1}+\alpha)}{m^d}=\lim_{m\rightarrow \infty}\frac{Q(mn_1,\ldots,m n_{r-1})}{m^d}.
$$
There exists $\alpha\in \ZZ_+$ such that $I(r-1)^{\alpha}\subset I(r)$. Thus for $n_1,\ldots,n_{r-1}\in \ZZ_+$,
$$
Q(n_1,\ldots,n_{r-1})\le G(n_1,\ldots,n_{r-1},1)\le Q(n_1,\ldots,n_{r-1}+\alpha)
$$
and thus we have equality of polynomials
$$
\begin{array}{lll}
Q(n_1,\ldots,n_{r-1})&=&\lim_{m\rightarrow \infty}\frac{Q(mn_1,\ldots,mn_{r-1})}{m^d}\\
&=&\lim_{m\rightarrow\infty}\frac{G(mn_1,\ldots,mn_{r-1},0)}{m^d}\\
&=&G(n_1,\ldots,n_{r-1},0)
\end{array}
$$
and the theorem holds (for $m_R$-primary ideals).
\end{proof}

As a consequence of the above proposition, we  extend the conclusions  of Theorem \ref{Theorem2} to all $n_1,\ldots,n_r\in \NN$.

\begin{Theorem}\label{Cor2} Suppose that $R$ is a Noetherian local ring of dimension $d$ such that 
$$
\dim N(\hat R)<d
$$
 and $\mathcal I(1)=\{I(1)_i\},\ldots,\mathcal I(r)=\{I(r)_i\}$ are (not necessarily Noetherian) filtrations of $R$ by $m_R$-primary ideals. Suppose that $M$ is a finitely generated $R$-module. Then there exists a homogeneous polynomial $G(x_1,\ldots,x_r)\in \RR[x_1,\ldots,x_r]$ which is of total degree $d$  if $G$ is nonzero, such that for all $n_1,\ldots,n_r\in \NN$,
$$
\lim_{m\rightarrow\infty}\frac{\lambda(M/I(1)_{mn_1}\cdots I(r)_{nm_r}M)}{m^d}=G(n_1,\ldots,n_r). 
$$
 \end{Theorem}

The proof of the following proposition is by the same method as  the proof of Theorem \ref{Theorem12}, starting with the fact that the identities of  Proposition \ref{exact} hold for $m_R$-primary ideals by   
\cite[lemma 17.4.4]{HS}.

\begin{Proposition}\label{exact}  Suppose that $R$ is a Noetherian local ring of dimension $d$ such that $\dim N(\hat R)<d$ and $\mathcal I(1)=\{I(1)_i\},\ldots,\mathcal I(r)=\{I(r)_i\}$ are  filtrations of $R$ by $m_R$-primary ideals. Suppose that 
$$
0\rightarrow M_1\rightarrow M_2\rightarrow M_3\rightarrow 0
$$
is a short exact sequence of finitely generated $R$-modules. Then for any $d_1,\ldots,d_r\in \NN$ with $d_1+\cdots+d_r=d$, we have that
$$
\begin{array}{l}
e_R(\mathcal I(1)^{[d_1]},\ldots,\mathcal I(r)^{[d_r]};M_2)\\
=e_R(\mathcal I(1)^{[d_1]},\ldots,\mathcal I(r)^{[d_r]};M_1)+ e_R(\mathcal I(1)^{[d_1]},\ldots,\mathcal I(r)^{[d_r]};M_3).
\end{array}
$$
\end{Proposition}

The following Associativity Formula is proven for $m_R$-primary ideals in  \cite[Theorem 17.4.8]{HS}.

\begin{Theorem} (Associativity Formula)\label{Theorem11} Suppose that $R$ is a Noetherian local ring of dimension $d$ with $\dim N(\hat R)<d$. Suppose $\mathcal I(j)=\{I(j)_i\}$ for $1\le j\le r$ are filtrations of $R$ by $m_R$-primary ideals and  $M$ is a finitely generated $R$-module  and $\mathcal I(1)=\{I(1)_i\},\ldots \mathcal I(r)=\{I(r)_i\}$ are filtrations of $R$ by $m_R$-primary ideals.  Let $P$ be a minimal prime of $R$. Then $\dim N(\widehat{R/P})<d$.
For any $d_1,\ldots,d_r\in \NN$ with $d_1+\cdots+d_r=d$, 
$$
e_R(\mathcal I(1)^{[d_1]},\ldots,\mathcal I(r)^{[d_r]};M)=
\sum\lambda_{R_P}(M_P)e_{R/P}((\mathcal I(1)R/P)^{[d_1]},\ldots,(\mathcal I(r)R/P)^{[d_r]};R/P)
$$
where the sum is  over the minimal primes of $R$ such that $\dim R/P=d$ and $\mathcal I(j)R/P=\{I(j)_iR/P\}$.
\end{Theorem}

\begin{proof} Let $\overline R=R/N(R)$. We have that $N(\widehat{\overline R})=N(\hat R)\widehat{\overline R}$ so $\dim N(\widehat{\overline  R})<d=\dim \widehat{\overline R}$.

Let $P_1,\ldots,P_s$ be the minimal primes of $R$ and $S=\bigoplus_{i=1}^sR/P_i$. As in the proof of Lemma \ref{Lemma5}, we have a natural inclusion $\overline R\rightarrow S$, and there exists a non zero divisor $x\in \overline R$ such that $xS\subset \overline R$. Further, $x$ is a non zerodivisor on $S$ since $S$ is a subring of the total quotient ring of $\overline R$. 
 Since completion is flat, we have an induced inclusion
 $$
 \widehat{\overline R}\rightarrow \hat S=\bigoplus_{i=1}^s \widehat{\overline R/P_i}.
 $$
 We have that $xN(\hat S)\subset N(\widehat{\overline R})$. Now $x$ is a non zero divisor on $\hat S$ since it is on $S$ and completion is flat. Thus $\dim N(\hat S)\le \dim N(\widehat{\overline R})<d$, and so
 $\dim N(\widehat{R/P_i})<d$ for all $i$.
 
 By Theorem \ref{Theorem2} and  Lemmas \ref{Lemma3} and  \ref{Prop1}, we have that 
 $$
 e_R(\mathcal I(1)^{[d_1]},\ldots,\mathcal I(r)^{[d_r]};M)=e_{\overline R}(\overline{\mathcal I}(1)^{[d_1]},\ldots,\overline{\mathcal I}(r)^{[d_r]};\overline M)
 $$
 where $\overline{\mathcal I}(j)=\{I(j)_i\overline R\}$ for $1\le j\le r$ and $\overline M=M/N(R)M$.
 
 By Theorem \ref{Theorem2} and Lemmas \ref{Lemma5} and  \ref{Prop1},
 $$
 e_{\overline R}(\overline{\mathcal I}(1)^{[d_1]},\ldots,\overline{\mathcal I}(r)^{[d_r]};\overline M)
 =\sum_{i=1}^se_{R/P_i}((\mathcal I(1)R/P_i)^{[d_1]},\ldots,(\mathcal I(r)R/P_i)^{[d_r]};M/P_iM).
 $$
 Now for $1\le i\le r$,
 $$
 \begin{array}{l}
 e_{R/P_i}((\mathcal I(1)R/P_i)^{[d_1]},\ldots,(\mathcal I(r)R/P_i)^{[d_r]};M/P_iM)\\
 =\lambda_{R_{P_i}}(M_{P_i})e_{R/P_i}((\mathcal I(1)R/P_i)^{[d_1]},\ldots,(\mathcal I(r)R/P_i)^{[d_r]};R/P_i)
 \end{array}
 $$
 by Lemma \ref{Lemma4}, since $R_{P_i}=Q(R/P_i)$.
\end{proof}

The following theorem generalizes  \cite[Proposition 11.2.1]{HS} for $m_R$-primary ideals to filtrations of $R$ by $m_R$-primary ideals.

\begin{Theorem}\label{Theorem13} Suppose that $R$ is a Noetherian $d$-dimensional local ring such that 
$$
\dim N(\hat R)<d
$$
 and $M$ is a finitely generated $R$-module. Suppose $\mathcal I'=\{I'_i\}$ and $\mathcal I=\{I_i\}$ are filtrations of $R$ by $m_R$-primary ideals. Suppose $\mathcal I'\subset \mathcal I$ ($
I'_i\subset I_i$ for all $i$) and the ring $\bigoplus_{n\ge 0}I_n$ is integral over $\bigoplus I'_n$. Then 
$$
e_R(\mathcal I;M)=e_R(\mathcal I';M).
$$
\end{Theorem}

The converse of Theorem \ref{Theorem13} is false. Taking $R$ to be a power series ring $R=k[[x_1,\ldots,x_d]]$ over a field $k$, let $I_i=m_R^i$ and $I'_i=m_R^{i+1}$. Then $e_R(\mathcal I;R)=e_R(\mathcal I';R)$
but $\bigoplus_{n\ge 0}I_n$ is not integral $\bigoplus_{n\ge 0}I'_n$. This is in contrast to a theorem of Rees, in \cite{R} and  \cite[Theorem 11.3.1]{HS}, showing that if 
$R$ is a formally equidimensional Noetherian local ring and $I'\subset I$ are $m_R$-primary ideals  then $\bigoplus_{n\ge 0}I^n$ is integral over $\bigoplus_{n\ge 0}(I')^n$ if and only if $e_R(I;R)=e_R(I';R)$.

\begin{proof}(of Theorem \ref{Theorem13}) Step 1). We first observe that if $I'\subset I$ are $m_R$-primary ideals and $\bigoplus_{n\ge 0}I^n$ is integral over $\bigoplus_{n\ge 0}(I')^n$, then, by \cite[Theorem 8.2.1, Corollary 1.2.5 and Proposition 11.2.1]{HS},   $e_R(I;R)=e_R(I';R)$.

Step 2).  Suppose $\mathcal I=\{I_i\}$ and $\mathcal I'=\{I'_i\}$ are Noetherian filtrations of $R$ by $m_R$-primary ideals and $\mathcal I'\subset \mathcal I$. Suppose $b\in \ZZ_+$. 
Define $\mathcal I^{(b)}=\{I^{(b)}_i\}$ where $I^{(b)}_i=I_{bi}$ and $(\mathcal I')^{(b)}=\{(I')^{(b)}_i\}$ where $(I')^{(b)}_i=(I')_{bi}$.
Then from Lemma \ref{Lemma0}  we deduce that
$$
e_R(\mathcal I;R)=e_R(\mathcal I';R)\mbox{ if and only if }e_R(\mathcal I^{(b)};R)=e_R((\mathcal I')^{(b)};R).
$$

Step 3). Suppose $\mathcal I'\subset \mathcal I$ are filtrations of $R$ by $m_R$-primary ideals. Suppose $a\in \ZZ_+$. Let $\mathcal I_a=\{I_{a,n}\}$ be the $a$-th truncated filtration of $\mathcal I$ defined in Definition \ref{trunc}. Then there exists $\overline a\in \ZZ$ such that every element of $\bigoplus_{n\ge 0}I_{a,n}$ (considered as a subring of $\bigoplus_{n\ge 0}I_n$) is integral over $\bigoplus_{n\ge 0}I'_{\overline a,n}$, where $\mathcal I'_{\overline a}=\{I'_{\overline a,i}\}$ is the $\overline a$-th truncated filtration of $\mathcal I'$  defined in Definition \ref{trunc}.
 
 Define a Noetherian filtration $\mathcal A_a=\{A_{a,i}\}$ of $R$ by $m_R$-primary ideals 
 by $A_{a,i}=I_{a,i}+I'_{\overline a,i}$. Thus we have inclusions of graded rings $\bigoplus_{n\ge 0}I'_{\overline a,n}\subset \bigoplus_{n\ge 0}A_{a,n}$ and $\bigoplus_{n\ge 0}A_{a,n}$ is finite over $\bigoplus_{n\ge 0}I'_{\overline a,n}$. By Steps 2) and 1),
$$
e_R(\mathcal I'_{\overline a};R)=e_R(\mathcal A_a;R).
$$
By Proposition \ref{Prop2},
$$
\lim_{a\rightarrow \infty}e_R(\mathcal I'_{\overline a};R)=e_R(\mathcal I';R)
$$
and thus 
$$
\lim_{a\rightarrow \infty}e_R(\mathcal A_a;R)=e_R(\mathcal I';R).
$$

Step 4) Let notation be as in the proof of Proposition \ref{Prop2}, but taking $J_i=I_i$ and $J(a)_i=I'_{ai}$. Define 
$$
\begin{array}{lll}
\Gamma(\mathcal A_a)^{(t)}&=&\{(m_1,\ldots,m_d,i)\in \NN^{d+1}\mid
\dim_kA_{a,i}\cap K_{m_1\lambda_1+\cdots+m_d\lambda_d}/A_{a,i}\cap K^+_{m_1\lambda_1+\cdots+m_d\lambda_d}\ge t\\
&&\mbox{ and }m_1+\cdots+m_d\le \beta i\}.
\end{array}
$$
Now $\Gamma(a)^{(t)}\subset \Gamma(\mathcal A_a)^{(t)}\subset \Gamma^{(t)}$ for all $t$, so
$$
\Delta(\Gamma(a)^{(t)})\subset \Delta(\Gamma(\mathcal A_a)^{(t)})\subset \Delta(\Gamma^{(t)})
$$
for all $a$. By (\ref{eq12}),
$$
\lim_{a\rightarrow \infty} {\rm Vol}(\Delta(\Gamma(a)^{(t)}))={\rm Vol}(\Delta(\Gamma^{(t)})),
$$
and so 
$$
\lim_{a\rightarrow \infty} {\rm Vol}(\Delta(\Gamma(A_a)^{(t)}))={\rm Vol}(\Delta(\Gamma^{(t)})).
$$
Thus
$$
\lim_{a\rightarrow \infty}e_R(\mathcal A_a;R)=e_R(\mathcal I;R)
$$
by (\ref{eq8}) of the proof of Proposition \ref{Prop2} applied to $\mathcal A_a$. 

Step 5). We have that $e_R(\mathcal I;R)=e_R(\mathcal I';R)$ by Steps 3) and 4). Now $e_R(\mathcal I;M)=e_R(\mathcal I';M)$ by Theorem \ref{Theorem11} (with $r=1$). 
\end{proof}

\begin{Corollary} Suppose $R$ is a Noetherian $d$-dimensional local ring such that 
$$
\dim N(\hat R)<d
$$
 and $M$ is a finitely generated $R$-module. Suppose that $\mathcal I(j)'=\{I(j)'_i\}$ and $\mathcal I(j)=\{I(j)_i\}$ are filtrations of $R$ by $m_R$-primary ideals for $1\le j\le r$. Suppose $\mathcal I(j)'\subset \mathcal I(j)$ for $1\le j\le r$ and the ring
$$
\bigoplus_{n_1,\ldots,n_r\ge 0}I(1)_{n_1}I(2)_{n_2}\cdots I(r)_{n_r}
$$
 is integral over 
 $$
 \bigoplus_{n_1,\ldots,n_r\ge 0}I(1)'_{n_1}I(2)'_{n_2}\cdots I(r)'_{n_r}.
 $$
  Then
\begin{equation}\label{eq15}
e_R(\mathcal I(q)^{[d_1]},\mathcal I(2)^{[d_2]},\ldots,\mathcal I(r)^{[d_r]};M)=
e_R((\mathcal I(1)')^{[d_1]},(\mathcal I(2)')^{[d_2]},\ldots,(\mathcal I(r)')^{[d_r]};M)
\end{equation}
for all $d_1,\ldots,d_r\in \NN$ with $d_1+\cdots+d_r=d$.
\end{Corollary}

\begin{proof} For $n_1,\ldots,n_r\in\ZZ_+$, the ring $\bigoplus_{m\ge 0}I(1)_{mn_1}I(2)_{mn_2}\cdots I(r)_{mn_r}$  is integral over $\bigoplus_{m\ge 0}I(1)'_{mn_1}I(2)'_{mn_2}\cdots I(r)'_{mn_r}$, so
$$
\lim_{m\rightarrow\infty}\frac{\lambda(M/I(1)_{mn_1}I(2)_{mn_2}\cdots I(r)_{mn_r}M)}{m^d}
=\lim_{m\rightarrow\infty}\frac{\lambda(M/I(1)'_{mn_1}I(2)'_{mn_2}\cdots I(r)'_{mn_r}M)}{m^d}
$$
by Theorem \ref{Theorem13}. Thus we have the equalities (\ref{eq15}) by Lemma \ref{Prop1} and Theorem \ref{Theorem2}.
\end{proof}

\section{Multigraded Filtrations}\label{Multi}

 We define a multigraded filtration $\mathcal I=\{I_{n_1,\ldots,n_r}\}_{n_1,\ldots,n_r\in\NN}$ of ideals on a ring $R$ to be a collection of ideals of $R$ such that
 $R=I_{0,\ldots,0}$, 
 $$
 I_{n_1,\ldots,n_{n_{j-1}},n_j+1,n_{j+1},\ldots,n_r}\subset I_{n_1,\ldots,n_{j-1},n_j,n_{j+1},\ldots,n_r}
 $$
  for all $n_1,\ldots,n_r\in\NN$ and
 $I_{a_1,\ldots,a_r}I_{b_1,\ldots,b_r}\subset I_{a_1+b_1,\ldots,a_r+b_r}$ whenever $a_1,\ldots,a_r,b_1,\ldots,b_r\in \NN$.
 
 A multigraded filtration
 $\mathcal I=\{I_{n_1,\ldots,n_r}\}$ of ideals on a local ring $R$ is a  multigraded filtration of $R$ by $m_R$-primary ideals if $I_{n_1,\ldots,n_r}$ is $m_R$-primary whenever $n_1+\cdots+n_r>0$.
 
 If $R$ is a Noetherian local ring of dimension $d$ with $\dim N(\hat R)<d$ and  $\mathcal I=\{I_{n_1,\ldots,n_r}\}$ is a  multigraded filtration of $R$ by $m_R$-primary ideals, then we can define (by Theorem \ref{TheoremI20}) the function 
 \begin{equation}\label{MG1}
P(n_1,\ldots,n_r)=\lim_{m\rightarrow \infty}\frac{\lambda(R/I_{mn_1,\ldots,mn_r})}{m^d}
\end{equation}
and ask if it has polynomial like behavior.  The following example shows that it can be far from polynomial like, so Theorem \ref{Theorem2} does not  have a good generalization to arbitrary multigraded filtrations of $m_R$-primary ideals. 

Let $R=k[[t]]$ be a power series ring over a field $k$. For $(n_1,n_2)\in \NN^2$, define $\alpha:\NN^2\rightarrow \NN$ by 
$$
\alpha(n_1,n_2)=\lceil \sqrt{n_1^2+n_2^2}\rceil
$$
where for a real number $x$,  $\lceil x \rceil$ is the smallest integer $a$ such that $x\le a$.

Define $I_{n_1,n_2}=(t^{\alpha(n_1,n_2)})$ and $\mathcal I=\{I_{n_1,n_2}\}$. Then $\mathcal I$ is a multigraded filtration of $R$ by $m_R$-primary ideals. For $(n_1,n_2)\in \NN^2$, we have that
$$
P(n_1,n_2)=\lim_{m\rightarrow \infty}\frac{\lambda(R/I_{mn_1,mn_2})}{m}=\lim_{m\rightarrow\infty}\frac{\lceil m\sqrt{n_1^2+n_2^2}\rceil}{m}=\lceil \sqrt{n_1^2+n_2^2}\rceil.
$$

We now  show that the function (\ref{MG1})  is polynomial like in an important situation. 
Let $R$ be an excellent, normal local ring of dimension two, and let $f:X\rightarrow\mbox{Spec}(R)$ be a resolution of singularities, with integral exceptional divisors $E_1,\ldots,E_r$. A resolution of singularities of a two dimensional, excellent local domain always exists by \cite{Li} or \cite{CJS}. If $n_1,\ldots,n_r\in \NN$, let $D_{n_1,\ldots,n_r}=\sum_{i=1}^rn_iE_i$, and define
$$
I_{n_1,\ldots,n_r}=\Gamma(X,\mathcal O_X(-D_{n_1,\ldots,n_r})),
$$
 which is an $m_R$-primary  ideal in $R$.  Then $\{I_{n_1,\ldots,n_r}\}$ is a multigraded filtration of $R$ by $m_R$-primary ideals. By Theorem 4 of \cite{C4}, if the divisor class group ${\rm Cl}(R)$ is not a torsion group, then there exists a resolution of singularities $f:X\rightarrow \mbox{spec}(R)$ and an exceptional divisor $F$ on $X$ such that $\bigoplus_{n\ge 0}\Gamma(X,\mathcal O_X(-nF))$ is not a finitely generated $R$-algebra, so that $\bigoplus_{n_1,\ldots,n_r\ge 0} I_{n_1,\ldots,n_r}$ is not a finitely generated $R$-algebra, and thus the multigraded filtration $\{I_{n_1,\ldots,n_r}\}$ is not Noetherian. In Proposition 6.3 \cite{CHR}, it is shown that there exists an  abstract complex of polyhedral sets $\mathcal P$ whose union is $\QQ_{\ge 0}$  (Definition 4.4 \cite{CHR}), such that for $P\in \mathcal P$ and $(n_1,\ldots,n_r)\in P\cap \NN^r$, 
 $$
 \lambda(R/I_{n_1,\ldots,n_r})=Q_P(n_1,\ldots,n_r)+L_P(n_1,\ldots,n_r)+\Phi_P(n_1,\ldots,n_r),
 $$
 where $Q_P(n_1,\ldots,n_r)$ is a quadratic polynomial with rational coefficients, $L_P(n_1,\ldots,n_r)$ is a linear function with periodic coefficients (a linear quasi polynomial) and $\Phi_P(n_1,\ldots,n_r)$ is a bounded function ($|\Phi_P(n_1,\ldots,n_r)|$ is bounded). Thus the function defined in (\ref{MG1}) is piecewise  polynomial, with
  $$
 P(n_1,\ldots,n_r)=\lim_{m\rightarrow \infty}\frac{\lambda(R/I_{mn_1,\ldots,mn_r})}{m^2}=Q_P(n_1,\ldots,n_r)
 $$
 if $(n_1,\ldots,n_r)\in P$. We have the further interpretation of  $P(n_1,\ldots,n_r)$ as  the intersection product
$$
P(n_1,\ldots,n_r)=-\frac{1}{2}(\Delta_{n_1,\ldots,n_r}^2)
$$
where $\Delta_{n_1,\ldots,n_r}$ is the  Zariski $\QQ$-divisor associated to  $-n_1E_1-\cdots-n_rE_r$ \cite[Formula (7)]{CHR}.


\begin{thebibliography}{1000000000}
\bibitem{Ab} S.S. Abhyankar, Resolution of Singularities of Embedded Algebraic Surfaces, Academic Press, New York (1966).
\bibitem{Bh} P.B. Bhattacharya, The Hilbert function of two ideals, Proc. Camb. Phil. Soc. 53 (1957), 568 - 575.
\bibitem{Bou} N. Bourbaki, Commutative Algebra, Chapters 1 - 7, Springer Verlag, 1989.
\bibitem{CJS} V. Cossart, U. Jannsen and S. Saito, Canonical embedded and non-embedded resolution of singularities  for excellent two-dimensional schemes, arXiv:0905.2191.
\bibitem{C4} S.D. Cutkosky, On unique and almost unique factorization of complete ideals II, Inventiones Math. 98 (1989), 59-74.
\bibitem{C1} S.D. Cutkosky, Multiplicities associated to graded families of ideals,  Algebra and Number Theory 7 (2013), 2059 - 2083.
\bibitem{C2} S.D. Cutkosky, Asymptotic multiplicities of graded families of ideals and linear series,  Advances in Mathematics 264 (2014), 55 - 113.
\bibitem{C3} S.D. Cutkosky, Asymptotic Multiplicities, Journal of Algebra 442 (2015), 260 - 298.
\bibitem{CHR} S.D. Cutkosky, J\"urgen Herzog and Ana Reguera, Poincar\'e series of resolutions of surface singularities, Transactions of the AMS 356 (2003), 1833 - 1874.
\bibitem{CS} S.D. Cutkosky and V. Srinivas, On a problem of Zariski on dimensions of linear systems, Annals Math. 137 (1993), 551 - 559.
\bibitem{ELS} L. Ein, R. Lazarsfeld and K. Smith, Uniform Approximation of Abhyankar valuation ideals in smooth function fields,
Amer. J. Math. 125 (2003), 409 - 440.
\bibitem{GGV} K. Goel, R.V. Gurjar and J.K. Verma, The Minkowski's (in)equality for multiplicity of ideals, preprint. 
 \bibitem{HHRT} M. Herrmann, E. Hyry, J. Rible and Z. Tang, Reduction numbers and multiplicities of multigraded structures, J. Algebra 17 (1997), 311- 341.
 \bibitem{H} J. Huh, Milnor numbers of projective hypersurfaes and the chromatic polynomial of graphs, J. Amer. Math. Soc. 25 (2012), 907 - 927.
\bibitem{J} N. Jacobson, Basic Algebra I, Second Edition, W.H. Freeman and Company, New York, 1985.
\bibitem{Ka} D. Katz, Note on multiplicity, Proc. Amer. Math. Soc. 104 (1988), 1021 - 1026.
\bibitem{KK} K. Kaveh and G. Khovanskii, Newton-Okounkov bodies, semigroups of integral points, graded algebras and intersection theory,  Annals of Math. 176 (2012), 925 - 978.
\bibitem{KK1} K. Kaveh and G. Khovanskii, Convex Bodies and Multiplicities of Ideals, Proc. Steklov Inst. Math. 286 (2014), 268 - 284.
\bibitem{KV} D. Katz and J. Verma, Extended Rees algebras and mixed multiplicities, Math. Z. 202 (1989), 111-128.
\bibitem{L} S. Lang, Algebra, revised third edition, Springer-Verlag, New York, Berlin, Heidelberg (2002).
\bibitem{Li} J. Lipman, Desingularization of 2-dimensional schemes, Annals of Math. 107 (1978), 115 - 207.
\bibitem{LM} R. Lazarsfeld and M. Musta\c{t}\u{a}, Convex bodies associated to linear series, Ann. Sci. Ec. Norm. Super 42 (2009) 783 - 835.
\bibitem{Ma2} H. Matsumura, Commutative Algebra, second edition, Benjamin/Cummings, Reading Massachusetts (1980).
\bibitem{MS} H. Muhly and M. Sakuma, Asymptotic factorization of ideals, J. London Math. Soc. 38 (1963), 341 - 350.
\bibitem{Mus} M. Musta\c{t}\u{a}, On multiplicities of graded sequence of ideals, J. Algebra 256 (2002), 229-249.
\bibitem{N} M. Nagata, Local Rings, Wiley, 1962.
\bibitem{Ok} A. Okounkov, Why would multiplicities be log-concave?, in The orbit method in geometry and physics, Progr. Math. 213, 2003, 329-347.
\bibitem{R} D. Rees, $\mathcal A$-transforms of local rings and a theorem on multiplicities of ideals, Proc. Cambridge Philos. Soc. 57 (1961), 8 - 17.
\bibitem{R1} D. Rees, Multiplicities, Hilbert functions and degree functions. In Commutative algebra: Durham 1981
(Durham 1981), London Math. Soc. Lecture Note Ser. 72, Cambridge, New York, Cambridge Univ. Press, 1982, 170 - 178.
\bibitem{RS} D. Rees and R. Sharp, On a Theorem of B. Teissier on Multiplicities of Ideals in Local Rings, J. London Math. Soc. 18 (1978), 449-463.
\bibitem{S} I. Swanson, Mixed multiplicities, joint reductions and a theorem of Rees, J. London Math. Soc. 48 (1993), 1 - 14.
\bibitem{HS} I. Swanson and C. Huneke, Integral Closure of Ideals, Rings and Modules, Cambridge University Press, 2006.
\bibitem{T1} B. Teissier Cycles \'evanescents,sections planes et condition de Whitney, Singularti\'es \`a Carg\`ese 1972, Ast\'erique 7-8 (1973)
\bibitem{T2} B. Teissier, Sur une in\'egalit\'e pour les multipliciti\'es (Appendix to a paper by D. Eisenbud and H. Levine),
Ann. Math. 106 (1977), 38 - 44.
\bibitem{T3} B. Teissier, On a Minkowski type  inequality for multiplicities II, In C.P. Ramanujam - a tribute, Tata Inst. Fund. Res. Studies in Math. 8, Berlin - New York, Springer, 1978.
\bibitem{TV} N.V. Trung and J. Verma, Mixed multiplicities of ideals versus mixed volumes of polytopes, Trans. Amer. Math. Soc. 359 (2007), 4711 - 4727.
\bibitem{ZS2} O. Zariski and P. Samuel, Commutative Algebra, Volume II,  Van Nostrand, Princeton (1960).

\end{thebibliography}
\end{document}